\documentclass[reqno,a4paper,12pt]{amsart}

\newcommand\version{March 9, 2023}

\usepackage{amsmath,amsfonts,amsthm,amssymb,amsxtra}
\usepackage{bbm} 
\usepackage{color}
\usepackage{hyperref}
\usepackage[normalem]{ulem}

\newtheorem{theorem}{Theorem}
\newtheorem{proposition}[theorem]{Proposition}
\newtheorem{lemma}[theorem]{Lemma}
\newtheorem{corollary}[theorem]{Corollary}

\theoremstyle{definition}

\theoremstyle{remark}

\newtheorem{remark}[theorem]{Remark}

\newcommand{\const}{\mathrm{const}\ }

\renewcommand{\epsilon}{\varepsilon}

\renewcommand{\phi}{\varphi}
\newcommand{\R}{\mathbb{R}}

\newcommand{\Sph}{\mathbb{S}}

\newcommand{\Z}{\mathbb{Z}}

\DeclareMathOperator{\dist}{dist}

\DeclareMathOperator{\dom}{dom}

\def\ce{\mathcal{E}}

\def\cL{\mathcal{L}}

\newcommand{\me}[1]{\mathrm{e}^{#1}}
\newcommand{\one}{\mathbf{1}}

\begin{document}

\date{\version}

\title[Hardy operators in a half-space]{On Sobolev norms involving\\ Hardy operators in a half-space}

\author[R. L. Frank]{Rupert L. Frank}
\address[Rupert L. Frank]{Mathematisches Institut, Ludwig-Maximilans Universit\"at M\"unchen, The\-resienstr. 39, 80333 M\"unchen, Germany, and Munich Center for Quantum Science and Technology (MCQST), Schellingstr. 4, 80799 M\"unchen, Germany, and Mathematics 253-37, Caltech, Pasa\-de\-na, CA 91125, USA}
\email{r.frank@lmu.de}

\author[K. Merz]{Konstantin Merz}
\address[Konstantin Merz]{Institut f\"ur Analysis und Algebra, Technische Universit\"at Braunschweig, Universit\"atsplatz 2, 38106 Braun\-schweig, Germany, and Department of Mathematics, Graduate School of Science, Osaka University, Toyonaka, Osaka 560-0043, Japan}
\email{k.merz@tu-bs.de}

\begin{abstract}
  We consider Hardy operators on the half-space, that is, ordinary and fractional Schr\"odinger operators with potentials given by the appropriate power of the distance to the boundary. We show that the scales of homogeneous Sobolev spaces generated by the Hardy operators and by the fractional Laplacian are comparable with each other when the coupling constant is not too large in a quantitative sense. Our results extend those in the whole Euclidean space and rely on recent heat kernel bounds.
\end{abstract}

\thanks{\copyright\, 2023 by the authors. This paper may be reproduced, in its entirety, for non-commercial purposes.\\
	Support through U.S. National Science Foundation grant DMS-1954995, by the Deutsche Forschungsgemeinschaft through Germany's Excellence Strategy, grant EXC-2111-390814868 (R.L.F.), and by the PRIME programme of the German Academic Exchange Service (DAAD) with funds from the German Federal Ministry of Education and Research (BMBF) (K.M.) is acknowledged.}

\maketitle
\tableofcontents

\section{Introduction and main result}

\subsection{Setting of the problem}

In this paper we consider the Hardy operators in a half-space, given informally by
\begin{equation}
	\label{eq:hardyops}
	L_\lambda^{(\alpha)} = (-\Delta)^{\alpha/2}_{\R^d_+} + \lambda x_d^{-\alpha}
	\qquad
	\text{in}\ L^2(\R^d_+) \,.
\end{equation}
Here and in what follows $\R^d_+=\R^{d-1}\times (0,\infty)$ and we write $x=(x',x_d)\in\R^{d-1}\times (0,\infty)$.

We are mostly interested in the fractional case $\alpha\in(0,2)$, but our results are also new in the local case $\alpha=2$. The operators $L_\lambda^{(\alpha)}$ are considered with a Dirichlet boundary condition for $\alpha=2$ and a certain analogue for $\alpha<2$. The precise meaning of $(-\Delta)^{\alpha/2}_{\R^d_+}$ will be explained in the next subsection; it is sometimes called the regional fractional Laplacian; see, e.g., \cite{Bogdanetal2003} and \cite[Section~8.4]{Kwasnicki2019}. 

The constant $\lambda$ is assumed to satisfy
$$
\lambda\geq \lambda_*
$$
where
$$
\lambda_* := -\frac{\Gamma(\frac{1+\alpha}{2})}{\pi} \left( \Gamma(\tfrac{1+\alpha}{2}) - \frac{2^{\alpha-1}\sqrt\pi}{\Gamma(1-\frac\alpha 2)} \right).
$$
Note that $\lambda_*$ depends on $\alpha$, but not on $d$, and that $\lambda_*=-\frac14$ if $\alpha=2$. Also, $\lambda_*<0$ if $\alpha\in(0,1)\cup(1,2]$ and $\lambda_*=0$ if $\alpha=1$.

The constant $\lambda_*$ plays the role of a critical coupling constant. As is well-known for $\alpha=2$ and shown by Bogdan and Dyda \cite{BogdanDyda2011} for $\alpha<2$, the constant $\lambda_*$ is the optimal constant in Hardy's inequality, which states that
$$
L_{\lambda_*}^{(\alpha)}\geq 0 \,.
$$
Our goal in this paper is to study the powers
$$
\left( L_\lambda^{(\alpha)} \right)^{s/2}
\qquad\text{with}\ s\in(0,2] \,.
$$
More precisely, we are interested in the domains of these operators (which are subspaces containing the operator domain of $L_\lambda^{(\alpha)}$) and, in particular, in the question how these domains for general $\lambda\geq\lambda_*$ compare with the domain of this operator in the case $\lambda=0$. When $\lambda>\lambda_*\neq 0$ and $s\leq 1$, it is easy to see that the domains of $( L_\lambda^{(\alpha)} )^{s/2}$ and $( L_0^{(\alpha)})^{s/2}$ coincide; see, e.g., \cite[Remark 1.2]{Franketal2021} for a similar argument. Our main interest is therefore in the case $s>1$, corresponding to subspaces between the form domain and the operator domain. In our main result (Theorem \ref{equivalencesobolev} below) we will show that, for a certain explicit range of $s$, depending on $\lambda$, the domains of $( L_\lambda^{(\alpha)} )^{s/2}$ and $( L_0^{(\alpha)})^{s/2}$ coincide.

There are several motivations for studying this question, coming both from pure mathematics and from applications to nonlinear dispersive equations and mathematical physics, and we will discuss some of them in Subsection \ref{s:motivation} below. There, we will also give references to the growing literature on the analogous question in other settings. Pioneering papers on this topic are those by Killip, Visan and Zhang \cite{Killipetal2016} and by Killip, Miao, Visan, Zhang and Zheng \cite{Killipetal2018}.


\subsection{Main result}

Before presenting our results, we will first discuss the definition of the operators \eqref{eq:hardyops} and then introduce a parameterization of the coupling constant $\lambda$ that will be important in what follows.

\subsubsection*{Definition of the operators}

Let us give the precise definition of $L_\lambda^{(\alpha)}$ as selfadjoint, nonnegative operators in the Hilbert space $L^2(\R^d_+)$. Throughout this paper, we assume that $d\geq 1$, $\alpha\in(0,2]$ and $\lambda\in[\lambda_*,\infty)$, except where explicitly stated otherwise.

For $\alpha\in(0,2)$, we consider the quadratic form
\begin{align*}
	\tfrac12\, \mathcal A(d,-\alpha) \iint_{\R_+^d\times\R_+^d} \frac{|u(x)-u(y)|^2}{|x-y|^{d+\alpha}}\,dx\,dy + \lambda \int_{\R^d_+} \frac{|u(x)|^2}{x_d^\alpha}\,dx
\end{align*}
with
\begin{align}
	\label{eq:adalpha}
	\mathcal A(d,-\alpha) := \frac{\alpha}{2^{1-\alpha}\pi^{d/2}}\ \frac{\Gamma(\frac{d+\alpha}2)}{\Gamma(1-\frac \alpha2)} \,,
\end{align}
and for $\alpha=2$ we consider the quadratic form
$$
\int_{\R^d_+} |\nabla u(x)|^2\,dx + \lambda \int_{\R^d_+} x_d^{-2}|u(x)|^2\,dx \,.
$$
These quadratic forms are considered for functions $u\in C^1_c(\R^d_+)$, that is, continuously differentiable functions whose support is a compact subset of the open set $\R^d_+$. According to the classical Hardy inequality for $\alpha=2$ and its sharp extension to $\alpha<2$ by Bogdan and Dyda \cite{BogdanDyda2011}, these quadratic forms are nonnegative if (and only if) $\lambda\in[\lambda_*,\infty)$. By a theorem of Friedrichs these forms therefore give rise to selfadjoint, nonnegative operators $L_\lambda^{(\alpha)}$ in $L^2(\R^d_+)$ for which $C^1_c(\R^d_+)$ is a form core.

The operators $(L_\lambda^{(\alpha)})^{s/2}$ appearing below are defined by the spectral theorem. We will use the fact that $C^\infty_c(\R^d_+)$ belongs to the domain of these operators for any $s\in [0,2]$ and any $\alpha\in(0,2]$; see Lemma \ref{domain}.


\subsubsection*{Definition of the exponent $p$}

For given $\alpha\in(0,2]$ (not reflected in the notation) we set $M:=\alpha$ if $\alpha<2$ and $M:=\infty$ if $\alpha=2$ and introduce the function
\begin{align}
	\label{eq:defcp}
	\begin{split}
		(-1,M)\ni p \mapsto C(p)
		& := \frac1\pi \left( \Gamma(\alpha)\, \sin\frac{\pi\alpha}2 + \Gamma(1+p)\,\Gamma(\alpha-p)\, \sin\frac{\pi(2p-\alpha)}{2} \right).
	\end{split}
\end{align}
When $\alpha=2$, one sees that the poles of $\Gamma(\alpha-p)$ cancel with the zeros of $\sin\frac{\pi(2p-\alpha)}{2}$ and, indeed, that $C(p)=p(p-1)$ for all $p>-1$. Similarly, for $\alpha=1$ one finds $C(p) = \frac1\pi (1-\pi p\cot\pi p)$.

The following properties of $C$ are known and we refer to Appendix \ref{s:defp} for details and references. The function $p\mapsto C(p)$ is continuous and symmetric with respect to $p=\frac{\alpha-1}2$, strictly increasing on $[\frac{\alpha-1}2,M)$ and its value at $p=\frac{\alpha-1}2$ is $\lambda_*$. Moreover, $\lim_{p\to M} C(p) = +\infty$. Thus, for any $\lambda\in[\lambda_*,\infty)$ there is a unique
\begin{align}
	\label{eq:defp}
	p \in [\tfrac{\alpha-1}2,M)
	\quad \text{with} \quad
  C(p) = \lambda \,.
\end{align}
We emphasize that $p$ depends on $\alpha$, besides $\lambda$.

One can show that $C(\alpha-1)=C(0)=0$. Thus, the case $\lambda=0$ corresponds to $p=(\alpha-1)_+:=\max\{\alpha-1,0\}$ and the case $\lambda>0$ to $p>(\alpha-1)_+$.

Using the explicit expression of $C(p)$ for $\alpha=2$ we see that
\begin{align}
	\label{eq:defp2}
	p =\tfrac12 \left( 1+\sqrt{1+4\lambda}\right)
	\qquad \text{if}\ \alpha=2 \,.
\end{align}

\subsubsection*{Notation}

We write
$$
A\wedge B := \min\{A,B\} \,,
\qquad
A\vee B := \max\{A,B\} \,.
$$
Moreover, in order to abbreviate some statements we suppress constants and write $A\lesssim B$ for $A,B\in\R_+$ whenever there is a constant $c>0$ such that $A\leq cB$. The notation $A\sim B$ means $A\lesssim B\lesssim A$ and in this case we say that $A$ and $B$ are comparable. If we want to emphasize that the constant $c$ may depend on some parameter, say $\tau$, we write $A\lesssim_\tau B$.

\subsubsection*{Main result -- Equivalence of Sobolev norms}

Our main result is contained in the following theorem. It states that the $L^2(\R_+^d)$-norms generated by certain powers of $L_\lambda^{(\alpha)}$ are comparable to those generated by the corresponding powers of $L_0^{(\alpha)}$.

\begin{theorem}
  \label{equivalencesobolev}
  Let
  $\alpha\in(0,2]$ and let $\lambda\geq0$ when $\alpha<2$ and $\lambda\geq-1/4$ when $\alpha=2$. Let $p$ be defined by \eqref{eq:defp}, and let $s\in(0,2]$.
  \begin{enumerate}
  \item If $s<(1+2p)/\alpha$, then $\dom (L_\lambda^{(\alpha)})^{s/2} \subset \dom (L_0^{(\alpha)})^{s/2}$ and 
    \begin{align}\label{eq:equivalencesobolev1}
      \| (L_0^{(\alpha)})^{s/2}u\|_{L^2(\R_+^d)} \lesssim_{d,\alpha,\lambda,s} \| (L_\lambda^{(\alpha)})^{s/2}u\|_{L^2(\R_+^d)}
      \quad \text{for all}\ u\in \dom (L_\lambda^{(\alpha)})^{s/2} \,.
    \end{align}
	Moreover, $C_c^\infty(\R^d_+)$ is an operator core of $(L_\lambda^{(\alpha)})^{s/2}$.
	
  \item If $s<(1+2(\alpha-1)_+)/\alpha$, then $\dom (L_0^{(\alpha)})^{s/2} \subset \dom (L_\lambda^{(\alpha)})^{s/2}$ and 
    \begin{align}\label{eq:equivalencesobolev2}
      \| (L_\lambda^{(\alpha)})^{s/2}u\|_{L^2(\R_+^d)} \lesssim_{d,\alpha,\lambda,s} \|(L_0^{(\alpha)})^{s/2}u\|_{L^2(\R_+^d)}
      \quad \text{for all}\ u\in \dom (L_0^{(\alpha)})^{s/2} \,.
    \end{align}
	Moreover, $C_c^\infty(\R^d_+)$ is an operator core of $(L_0^{(\alpha)})^{s/2}$.
  \end{enumerate}
\end{theorem}

In particular, for $s\in(0,2]$ with $s< \frac{1+2(p\wedge(\alpha-1)_+)}{\alpha}$ we have the equality $\dom (L_\lambda^{(\alpha)})^{s/2} = \dom (L_0^{(\alpha)})^{s/2}$ as well as the equivalence
\begin{equation*}
	\| (L_\lambda^{(\alpha)})^{s/2}u\|_{L^2(\R_+^d)} \sim_{d,\alpha,\lambda,s} \|(L_0^{(\alpha)})^{s/2}u\|_{L^2(\R_+^d)}
	\quad \text{for all}\ u\in \dom (L_\lambda^{(\alpha)})^{s/2} \,.
\end{equation*}
Note also that 
$$
p\wedge(\alpha-1)_+ = \begin{cases} (\alpha-1)_+ & \text{if}\ \lambda\geq 0 \,, \\ p & \text{if}\ \lambda\leq 0 \,. \end{cases}
$$
In Section \ref{sec:proofmain} we will see that the assumption $s< \frac{1+2(p\wedge(\alpha-1)_+)}{\alpha}$ is necessary for the equality $\dom (L_\lambda^{(\alpha)})^{s/2} = \dom (L_0^{(\alpha)})^{s/2}$ (under the additional assumption $\alpha<3/2$ if $d=1$).

For $\alpha=1$, we have $\lambda_*=0$ and the assumption $\lambda\geq 0$ in Theorem \ref{equivalencesobolev} is optimal, as is the assumption $\lambda\geq-1/4$ for $\alpha=2$. For $\alpha\in(0,2)\setminus\{1\}$ the restriction to $\lambda\geq 0$ is probably technical. It comes from bounds on the heat kernel of $L_\lambda^{(\alpha)}$, which are an ingredient in our proofs and which are currently known only for $\lambda\geq0$ when $\alpha<2$. Since we expect these bounds to be true also for $\lambda\in[\lambda_*,0)$, we will accompany each of our main results with a remark stating the potential extension.

\begin{remark}
  Let $\alpha\in(0,2)$, $\lambda\in[\lambda_*,0)$ and assume that $\me{-tL_\lambda^{(\alpha)}}(x,y)$ satisfies the upper bound in \eqref{eq:heatkernel} below with $p$ defined by \eqref{eq:defp}. Then Theorem \ref{equivalencesobolev} remains valid for this value of $\lambda$. This follows by the same arguments as in the proof below, taking into account Remarks~\ref{genhardyrem}, \ref{reversehardyrem} and~\ref{densityrem}.
\end{remark}

We next present two important ingredients in the proof of Theorem \ref{equivalencesobolev} which are of independent interest. They concern variants of Hardy's inequality.

\begin{theorem}[Generalized Hardy inequality]
  \label{genhardy}
  Let
  $\alpha\in(0,2]$ and let
  $\lambda\geq0$ when $\alpha\in(0,2)$ and $\lambda\geq-1/4$ when $\alpha=2$.
  Let $p$ be defined by \eqref{eq:defp}. Then, if $s\in(0,\frac{1+2p}{\alpha} \wedge \frac{2d}\alpha)$, one has
  \begin{align}
    \label{eq:genhardy}
    \| x_d^{-\alpha s/2}u\|_{L^2(\R_+^d)}
    \lesssim_{d,\alpha,\lambda,s} \|(L_\lambda^{(\alpha)})^{s/2}u\|_{L^2(\R_+^d)}
    \qquad\text{for all}\ u\in C_c^\infty(\R_+^d) \,.
  \end{align}
\end{theorem}

\begin{remark}\label{genhardyrem}
  Let $\alpha\in(0,2)$, $\lambda\in[\lambda_*,0)$ and assume that $\me{-tL_\lambda}(x,y)$ satisfies the upper bound in \eqref{eq:heatkernel} below with $p$ defined by \eqref{eq:defp}. Then Theorem \ref{genhardy} remains valid for this value of $\lambda$. This follows by the same arguments as in the proof below, taking into account Remark \ref{genhardybddrem}.
\end{remark}

It is interesting to compare the assumption $s\in(0,\frac{1+2p}{\alpha} \wedge \frac{2d}\alpha)$ in Theorem \ref{genhardy} with the corresponding assumption for the Hardy inequality in $\R^d$ with weight $|x|^{-\alpha s/2}$ with a point singularity, namely $s\in(0,\frac{d+2p}{\alpha} \wedge \frac{2d}\alpha)$; cf.~\cite[Proposition~1.4]{Franketal2021} or \cite[Proposition~3.2]{Killipetal2018}. The difference between $d$ and $1$ in this assumption reflects the different dimensionalities of the sets where the Hardy weight is singular.

\begin{theorem}[Reversed Hardy inequality]
  \label{reversehardy}
  Let
  $\alpha\in(0,2]$ and let
  $\lambda\geq0$ when $\alpha<2$ and $\lambda\geq-1/4$ when $\alpha=2$. Let $p$ be defined by \eqref{eq:defp} and let $s\in(0,2]$. Then
  \begin{align*}
    \left\| \left( (L_{\lambda}^{(\alpha)})^{s/2} - (L_0^{(\alpha)})^{s/2} \right) u \right\|_{L^2(\R_+^d)}
    \lesssim_{d,\alpha,\lambda,s} \| x_d^{-\alpha s/2}u\|_{L^2(\R_+^d)}
    \quad\text{for all}\ u\in C_c^\infty(\R_+^d) \,.
  \end{align*}
\end{theorem}

\begin{remark}
  \label{reversehardyrem}
  Let $\alpha\in(0,2)$, $\lambda\in[\lambda_*,0)$ and assume that $\me{-tL_\lambda^{(\alpha)}}(x,y)$ satisfies the upper bound in \eqref{eq:heatkernel} below with $p$ defined by \eqref{eq:defp}. Then Theorem \ref{reversehardy} remains valid for this value of $\lambda$. This follows by the same arguments as in the proof below, taking into account Remark \ref{differenceheatkernel3rem}.
\end{remark}


\begin{remark}\label{otherdef}
	We have made the choice to compare the operators $L_\lambda^{(\alpha)}$ for general $\lambda\geq\lambda_*$ with the operator $L_0^{(\alpha)}$ for the case $\lambda=0$. This is natural given the quadratic form definition of the operators $L_\lambda^{(\alpha)}$. For $\alpha<2$, there is another natural choice for the comparison operator, namely $L_{\lambda_0}^{(\alpha)}$ with $\lambda_0\in (0,\infty)$ defined by
	$$
	\mathcal A(d,-\alpha) \int_{\R_-^d} \frac{dy}{|x-y|^{d+\alpha}} = \frac{\lambda_0}{x_d^\alpha} \,.
	$$
	(The fact that the left side is a constant multiple of $x_d^{-\alpha}$ follows by simple translation and dilation considerations.) With this definition of $\lambda_0$, we have for $u\in C^1_c(\R^d_+)$, identified with its extension by zero to $\R^d$,
	\begin{align*}
		\| (-\Delta)^{\alpha/4} u \|_{L^2(\R^d)}^2 & = \tfrac12\, \mathcal A(d,-\alpha) \iint_{\R^d\times\R^d} \frac{|u(x)-u(y)|^2}{|x-y|^{d+\alpha}}\,dx\,dy \\
		& = \tfrac12\, \mathcal A(d,-\alpha) \iint_{\R_+^d\times\R_+^d} \frac{|u(x)-u(y)|^2}{|x-y|^{d+\alpha}}\,dx\,dy + \lambda_0 \int_{\R^d_+} \frac{|u(x)|^2}{x_d^\alpha}\,dx \,.
	\end{align*}
	In this sense the operator $L_{\lambda_0}^{(\alpha)}$ is equally natural as $L_{0}^{(\alpha)}$. Our arguments in this paper extend without significant changes to the case where we compare with $L_{\lambda_0}^{(\alpha)}$. However, for the sake of concreteness and conciseness we have decided to present the arguments in the case of comparison with the operator $L_0^{(\alpha)}$.
\end{remark}

\begin{remark}
  We consider the Schr\"odinger operators $L_\lambda^{(\alpha)}$ whose potential is precisely $\lambda x_d^{-\alpha}$. In some applications it is necessary to allow more general potentials $V$ satisfying $\lambda |x|^{-\alpha} \leq V(x) \leq \tilde \lambda x_d^{-\alpha}$ for all $x\in\R^d_+$ with some $\lambda_*\leq\lambda\leq\tilde\lambda<\infty$. In this case an analogue of Theorem \ref{equivalencesobolev} holds with $p$ defined by \eqref{eq:defp} with the given $\lambda$; in particular, it is independent of $\tilde\lambda$. This follows by a simple modification of our proofs. We have carried out the details in \cite[Section 4]{Franketal2021} in the case of Hardy weights with point singularities and omit the corresponding details here.
\end{remark}

\subsection{Background and motivation}
\label{s:motivation}

After having presented our main results, we would like to put them into context and discuss some previous, related results.

Homogeneous operators appear frequently in applications as model operators or as scaling limits of more complicated operators, and one aims at analyzing them in as much detail as possible to draw conclusions about the perturbed versions that appear in applications. From the point of view of pure mathematics and harmonic analysis homogeneous operators are interesting as testing grounds of how much of Euclidean Fourier analysis remains valid when one dispenses with translation invariance.

A typical feature of homogeneous operators is the appearance of critical coupling constants. These are often related to sharp constants in Hardy-type inequalities. For instance, Hardy's original inequality \cite{Hardy1919,Hardy1920,OpicKufner1990,Kufneretal2006} is the case $d=1$ of the inequality
$$
\int_{\R^d_+} |\nabla u|^2\,dx \geq \frac14 \int_{\R^d_+} \frac{|u|^2}{x_d^2}\,dx
\qquad\text{for all}\ u\in C^1_c(\R^d_+) \,.
$$
This inequality is precisely what guarantees that the operators $L_\lambda^{(2)}$ with $\lambda\geq-\frac14$ are lower semibounded on $C^1_c(\R^d_+)$ and therefore can be realized as selfadjoint operators in $L^2(\R^d_+)$. The fact that the constant $\frac14$ in Hardy's inequality is sharp means that the operators $L_\lambda^{(2)}$ with $\lambda<-\frac14$ are not lower semibounded on $C^1_c(\R^d_+)$ and therefore cannot have a lower bounded selfadjoint extension. In applications the operators $L_\lambda^{(2)}$ appear almost only with $\lambda\geq-\frac14$.

Another natural extension of Hardy's inequality to the higher dimensional case is
$$
\int_{\R^d} |\nabla u|^2\,dx \geq \frac{(d-2)^2}4 \int_{\R^d} \frac{|u|^2}{|x|^2}\,dx
\qquad\text{for all}\ u\in C^1_c(\R^d) \ \text{if}\ d\geq 3 \,.
$$
The corresponding operators $-\Delta + \lambda |x|^{-2}$ for $\lambda \geq - \frac{(d-2)^2}{4}$ where studied in the influential paper by Killip, Miao, Visan, Zhang and Zheng \cite{Killipetal2018}. These authors were motivated by the analysis of nonlinear dispersive PDEs, more precisely, by the study of the global well-posedness and scattering for the nonlinear Schr\"odinger equation with inverse-square potential \cite{Killipetal2017,Killipetal2017T}. In \cite{Killipetal2018} the domains of the operators $(-\Delta + \lambda |x|^{-2})^{s/2}$ were compared with the homogeneous Sobolev spaces $\dot H^{s}(\R^d)$ and in this connection a relation between the power $s$ and the coupling constant $\lambda$ was observed for the first time. Earlier, Killip, Visan and Zhang \cite{Killipetal2016} had studied a similar question for the Dirichlet Laplacian on the complement of a compact, convex set, motivated again by questions about nonlinear Schr\"odinger equations. The techniques developed in \cite{Killipetal2016,Killipetal2018} play an important role in our analysis.

Hardy's inequality has been generalized to powers of the Laplacian. A special case of a result by Herbst \cite{Herbst1977} is that
$$
\left\| (-\Delta)^{\alpha/4} u \right\|_{L^2(\R^d)}^2 \geq 2^{\alpha} \, \frac{\Gamma(\frac{d+\alpha}{4})^2}{\Gamma(\frac{d-\alpha}{4})^2}\, \left\| |x|^{-\alpha/2} u \right\|_{L^2(\R^d)}^2
\qquad\text{for all}\ u\in\dot H^\frac\alpha2(\R^d) \ \text{if}\ d>\alpha \,.
$$
For alternative proofs of Herbst's inequality see \cite{Kovalenkoetal1981,Yafaev1999,Franketal2008H,FrankSeiringer2008}. Of particular importance is the case $\alpha=1$ and $d=3$, since the operator $\sqrt{-\Delta+m^2}-m^2 +\lambda |x|^{-1}$ in $L^2(\R^3)$ provides a model for a relativistic description of an electron in the Coulomb field of a point nucleus. The scale invariant model problem for the latter operator is the homogeneous operator $\sqrt{-\Delta} +\lambda |x|^{-1}$ and many results about the latter operator have implications for the quantum mechanics with relativistic effects. For instance, Lieb--Thirring inequalities for the latter operator were used to solve the problem of stability of matter in the presence of magnetic fields \cite{Franketal2008H,Franketal2007S}.

Recently, in joint work with Heinz Siedentop and Barry Simon, we discussed the analogue of the strong Scott conjecture for relativistic electrons \cite{Franketal2020P}. This is a quantum many-body problem, where the underlying one-body operator is again $\sqrt{-\Delta+m^2}-m^2 +\lambda |x|^{-1}$ in $L^2(\R^3)$. In connection with this investigation we needed information about the domains of the operators $(\sqrt{-\Delta} +\lambda |x|^{-1})^{s/2}$. More precisely, in our approach we needed to know that for any $\lambda>\lambda_*$ there is an $s>1$ such that the $L^2(\R^d)$-norms generated by $(\sqrt{-\Delta} +\lambda |x|^{-1})^{s/2}$ are equivalent to those generated by $(-\Delta)^{s/2}$. That this is indeed the case was shown in \cite{Franketal2021}, thus leading to a proof of the strong Scott conjecture in the relativistic case. For an alternative proof see \cite{Franketal2023} and for a review about the Scott conjecture see \cite{Franketal2023T}.

In passing we mention that the papers \cite{Killipetal2016,Killipetal2018} also deal with the case where the underlying norms are those in $L^p(\R^d)$ with $p\neq 2$. Similarly, the results in \cite{Franketal2021}, which concerned $L^2$-norms, have been extended to $L^p$-norms with general $1<p<\infty$; see \cite{Merz2021} for $\lambda>0$ and \cite{BuiDAncona2023,BuiNader2022} for all $\lambda\geq\lambda_*$. Proofs for $p\neq 2$ often rely on multiplier theorems in the spirit of the Mikhlin--H\"ormander theorem. (Note that such multiplier theorems are immediate consequences of the spectral theorem when $p=2$.) In the local case $\alpha=2$ the proof of multiplier theorems can be based on heat kernel bounds with Gaussian off-diagonal decay. In the absence of such bounds the case $\alpha<2$ is substantially more complicated; see also \cite{Merz2022}.

\medskip

In the present paper we address the analogous question in the $L^2$-case for fractional operators on half-spaces. The corresponding sharp Hardy inequality in this setting is due to Bogdan and Dyda \cite{BogdanDyda2011} and states that
\begin{align*}
	\tfrac12\, \mathcal A(d,-\alpha) \iint_{\R_+^d\times\R_+^d} \frac{|u(x)-u(y)|^2}{|x-y|^{d+\alpha}}\,dx\,dy \geq -\lambda_* \int_{\R^d_+} \frac{|u(x)|^2}{x_d^\alpha}\,dx
	\ \text{for all}\ u\in C^1_c(\R^d_+) \,.
\end{align*}
For an alternative proof see \cite{FrankSeiringer2010}.

The main new difficulty compared to previous investigations is the presence of a boundary in the fractional case. Note that there is an interplay between the order $\alpha$ of the operator and the effect of the boundary. For $\alpha<1$ we expect the influence of the boundary to be negligible, with $\alpha=1$ being a subtle borderline case. This expectation manifests itself, for instance, in the appearance of the positive part $(\alpha-1)_+$ in part~(2) of Theorem \ref{equivalencesobolev}. Related to this is the appearance, for small $\alpha$ and large $\lambda$, of a large extra factor in the Riesz kernel bounds (Theorem~\ref{riesz} below) when the distance of both points to the boundary is much smaller than their mutual distance. This is a phenomenon not encountered in previous studies of similar questions.

We expect our results in the model case of a homogeneous operator on a half-space to have applications and extensions to the study of both more general operators and more general domains.

\subsection{Method of proof and organization of the paper}

The proof of Theorem \ref{equivalencesobolev} consists of two parts. In the first part, we prove the relevant inequalities for functions in $C_c^\infty(\R^d_+)$ and in the second part, we show that the latter set is an operator core, thereby extending the inequalities to all functions in the domain in the relevant operators.

The first part of the proof of Theorem \ref{equivalencesobolev}, is an immediate consequence of Theorems \ref{genhardy} and \ref{reversehardy}. The main ingredient for the proof of both of these theorems are pointwise bounds on the heat kernels of the operators $L_\lambda^{(\alpha)}$, which have been proved recently by Cho, Kim, Song and Vondra\v{c}ek \cite{Choetal2020} and Song, Wu and Wu \cite{Songetal2022} for $\alpha<2$. The structure of these bounds is that they differ from the whole space heat kernel by a product of two extra factors that depend on the distance of $x$ (resp.\ $y$) from the boundary relative to $t^{1/\alpha}$. This is summarized in Section \ref{s:heatkernel}, with some technical details deferred to Appendix \ref{s:localheatkernel}.

For the proof of Theorem \ref{genhardy} we use these heat kernel bounds to deduce Riesz kernel bounds, that is, bounds on the kernels of the operators $(L_\lambda^{(\alpha)})^{-s/2}$ with $s<\frac{2d}\alpha$; see Theorem \ref{riesz}. For $\alpha=2$ and all $\lambda$, or for $\alpha<2$ and all not too large $\lambda$ (depending on $\alpha$ and $s$), these Riesz kernel bounds inherit the structure of the heat kernel bounds, namely the whole space kernel multiplied by two extra factors. When $\alpha<2$ and $\lambda$ is large, however, this product structure of the Riesz kernel bounds is no longer valid and needs to be replaced by a term, which relative to the product structure becomes unbounded when both $x$ and $y$ are close to the boundary (compared to $|x-y|$). This phenomenon does not occur in previous works on related questions, such as \cite{Killipetal2016,Killipetal2018,Franketal2021}. 

Once the Riesz kernel bounds have been established, the generalized Hardy inequality in Theorem \ref{genhardy} follows by Schur tests; see Section \ref{s:genhardy}. This is conceptually similar to \cite{Killipetal2016,Killipetal2018,Franketal2021}, but the violation of the product structure for certain $\lambda$ necessitates some extra efforts. This will complete the proof of Theorem~\ref{genhardy}.

Turning to the proof of Theorem \ref{reversehardy}, we need bounds on the \emph{difference} of the heat kernels of $L_\lambda^{(\alpha)}$ and $L_0^{(\alpha)}$. Those are derived in Section \ref{s:differenceheatkernels}. The difficulty here is that in a certain region of space, namely when both $x$ and $y$ are far away from the boundary (compared to $t^{1/\alpha}$), but close together (compared to their distance from the boundary), one needs to quantify a cancellation coming from taking the \emph{difference} of the heat kernels. Again there are similarities to earlier such arguments, but we believe that here we carry them out more efficiently than in \cite{Franketal2021} and that our new arguments would simplify the proof in \cite{Franketal2021}.

Once the bounds on the difference of the heat kernels have been established, the reverse Hardy inequality in Theorem \ref{reversehardy} follows by Schur tests; see Section \ref{s:schur}. These Schur tests are again conceptually similar to earlier arguments, but require substantially more technical work.

It is perhaps worth pointing out the simple idea that guides the technical work in Sections \ref{s:genhardy}, \ref{s:differenceheatkernels} and \ref{s:schur}, namely to exploit the invariance of the operators $L_\lambda^{(\alpha)}$ with respect to translations parallel to the boundary. This implies that the kernels of the various operators discussed above depend on the variables $x'$ and $y'$ only through their difference $x'-y'$ (in fact, only on $|x'-y'|$), and therefore one aims at integrating out these variables. In this way we try to effectively reduce the problem to the one for the operator $L_\lambda^{(\alpha)}$ in one dimension. Once one is in one dimension, the distinction of the various regions (defined through the length scales $x_d$, $y_d$, $|x-y|$ and $t^{1/\alpha}$) simplifies considerably and allows one to conclude the proof.

We also note that we could have used the invariance with respect to translations parallel to the boundary already at the beginning and written $L_\lambda^{(\alpha)}$ as a direct integral of certain operators $L_\lambda^{(\alpha)}(\xi')$ in $L^2(\R_+)$, depending on a parameter $\xi'\in\R^{d-1}$, the Fourier variable corresponding to the space variable $x'$. In this way, we can rewrite all inequalities in Theorems \ref{equivalencesobolev}, \ref{genhardy} and \ref{reversehardy} as inequalities for the operators $L_\lambda^{(\alpha)}(\xi')$ \emph{with constants uniform in} $\xi'$. While this would have immediately reduced the problem to the one-dimensional case, one would have to deal with the uniformity in the parameter $\xi'$. Also, as far as we know, precise heat kernel bounds for the operators $L_\lambda^{(\alpha)}(\xi')$ are not available in the literature. (In this connection we mention the recent heat kernel bounds for $(-\Delta+1)^{\alpha/2} + V_\lambda^{(\alpha)}$ in $L^2(\R^d)$ for certain critical potentials $V_\lambda^{(\alpha)}$ that satisfy $V_\lambda^{(\alpha)} \sim \lambda |x|^{-\alpha}$ as $x\to 0$; see \cite{Jakubowskietal2022,Jakubowskietal2022R}.) We also note that precise information on the operators $(-\tfrac{d^2}{dx^2}+|\xi'|^2)^{\alpha/2}$ in $L^2(\R_+)$ (defined on $C^1_c(\R_+)$ via extension by zero to $\R$, then action on $\R$ and then restriction back to $\R_+$) has been obtained in \cite{Kwasnicki2011}. This information has been instrumental in \cite{FrankGeisinger2016}. These operators are similar, but in general different from the operators $L_\lambda^{(\alpha)}(\xi')$.

This concludes our discussion of the first part of the proof of Theorem \ref{equivalencesobolev}. The second part, namely the proof of the operator core property, takes up Sections \ref{sec:comm} and \ref{sec:density}. The main result here is Theorem \ref{density} in Section \ref{sec:density}. Its proof relies once more on the heat kernel bounds in Section \ref{s:heatkernel}. The novel ingredient here is a combination of these bounds with Schauder theory for the Laplacian and its fractional analogue. Applying Schauder estimates on appropriately chosen scales we obtain local H\"older norm bounds. These allow us to control action of the commutator of $(-\Delta)^{\alpha/2}$ with cut-off functions. Such bounds are the topic of Section \ref{sec:comm}.

We end this introduction by noting that in this paper we have restricted ourselves to the case where the underlying norms are $L^2$-norms. This is the case most frequently encountered in applications, including the before-mentioned ones to mathematical physics. There are other applications, such as those in connection with nonlinear Schr\"odinger equations, where one needs $L^p$-norm with general $1<p<\infty$. Also from a harmonic analysis point of view the proof of such bounds is a formidable problem, related to spectral multiplier theorems; see the references above in the case of a point singularity. Proving an analogue in the present situation of singularities along a hyperplane is an open problem.

\section{Heat kernel bounds for Hardy operators}
\label{s:heatkernel}

\textbf{Notation.}
In the following, we omit the superscript $(\alpha)$ in the notation for $L_\lambda^{(\alpha)}$ and write merely $L_\lambda\equiv L_\lambda^{(\alpha)}$ when there is no danger of confusion. 

\medskip


Of fundamental importance for us are pointwise bounds on the heat kernel of $L_\lambda$. We begin with the case $\alpha<2$.

\begin{theorem}
  \label{heatkernel}
  Let
  $\alpha\in(0,2)$ and let $\lambda\geq0$. Let $p$ be defined by \eqref{eq:defp}. Then one has, for all $x,y\in\R_+^d$ and $t>0$,
  \begin{align}
    \label{eq:heatkernel}
    e^{-tL_\lambda}(x,y) \sim \left( 1 \wedge \frac{x_d}{t^{1/\alpha}} \right)^{p} \left( 1 \wedge \frac{y_d}{t^{1/\alpha}} \right)^{p} t^{-\frac d{\alpha}} \left( 1 \wedge \frac{t^{1/\alpha}}{|x-y|} \right)^{d+\alpha}\,.
  \end{align}
\end{theorem}

Let us give references for where this theorem is proved. For $\lambda=0$ and $\alpha\leq 1$ the bound appears in \cite{ChenKumagai2003}. (More precisely, \cite{ChenKumagai2003} considers reflected processes, but for $\alpha\leq 1$ this coincides with the censored processes that we are interested in.) For $\lambda=0$ and $1<\alpha<2$ the bound appears in \cite{Chenetal2010}. (More precisely, \cite{Chenetal2010} only has this bound up to some arbitrary, but fixed time. However, by scaling invariance, once this bound is proved for any given time, it follows for all times.) The case $\lambda\geq0$ has been treated more recently and the bound appears in \cite{Choetal2020}; see also \cite{Songetal2022}.

Our definition of the function $p\mapsto C(p)$, which relates $p$ and the coupling constant $\lambda$, is seemingly different from the one used in \cite{Choetal2020}. We show that it is not in Appendix \ref{s:defp}.

As we have already said in the introduction, the restriction $\lambda\geq 0$ in our main results is a consequence of this restriction in Theorem \ref{heatkernel}. We expect that the latter theorem, and therefore also our main results, extend to the full range $\lambda\geq\lambda_*$.


We now turn the case $\alpha=2$.

\begin{theorem}\label{localheatkernelhardy}
  Let $\alpha=2$ and let $\lambda\geq-\frac14$. Let $p$ be given by \eqref{eq:defp}, that is, by \eqref{eq:defp2}. Then, for all $x,y\in\R_+^d$ and $t>0$,
	\begin{align}
		\label{eq:localheatkernelhardy}
		\exp\left(-tL_\lambda\right)(x,y)
		\asymp \left(1\wedge\frac{x_d}{\sqrt t}\right)^{p}\left(1\wedge\frac{y_d}{\sqrt t}\right)^{p} t^{-d/2}\me{-c|x-y|^2/t}\,,
	\end{align}
	where the notation $\asymp$ means the same as $\sim$, but where the constants $c$ in the exponential function are allowed to be different in the upper and the lower bounds.
\end{theorem}

While an explicit expression of the heat kernel of $L_\lambda$ for $\alpha=2$ is available, it leads to a somewhat different heat kernel bound and we explain in Appendix \ref{s:localheatkernel} how to obtain the bound stated in Theorem \ref{localheatkernelhardy}, where one is willing to give up something in the constant $c$ in the exponent, but insists on the product structure of the prefactor.

\section{Riesz kernel bounds}

In this section we use the heat kernel bounds from the previous section to prove two-sided bounds on the kernels of the Riesz operators $L_\lambda^{-s/2}$. They are crucial for the proof of the generalized Hardy inequality (Theorem~\ref{genhardy}).

\begin{theorem}
  \label{riesz}
  Let
  $\alpha\in(0,2]$ and let
  $\lambda\geq0$ when $\alpha\in(0,2)$ and $\lambda\geq-1/4$ when $\alpha=2$.
 Let $p$ be defined by \eqref{eq:defp} and let
  $s\in(0,\frac{2d}{\alpha}\wedge\frac{2(d+2p)}{\alpha})$. Then the following holds.
  \begin{enumerate}
  \item[(a)] For all $x,y\in\R^d_+$ with $|x-y|\leq x_d\vee y_d$,
  \begin{align}
      \label{eq:riesz}
      L_\lambda^{-s/2}(x,y)  \sim_{d,\alpha,\lambda,s} |x-y|^{\alpha\frac{s}{2}-d}\left(1\wedge\frac{x_d}{|x-y|} \wedge\frac{y_d}{|x-y|}\right)^{p}.
  \end{align}

  \item[(b)] For all $x,y\in\R^d_+$ with $x_d \vee y_d \leq |x-y|$,
    \begin{align}
      \label{eq:rieszimprovedsim}
      \begin{split}
        & L_\lambda^{-s/2}(x,y)
        \sim_{d,\alpha,\lambda,s} |x-y|^{\alpha\frac{s}{2}-d} \left(\frac{x_d\,y_d}{|x-y|^2}\right)^p \cdot \Biggl[\one_{\alpha=2} \\
        & \quad + \left(\one_{p\leq\frac{\alpha}{2}(1+\frac s2)} + \left(\ln\frac{|x-y|}{x_d\vee y_d}\right)\one_{p=\frac{\alpha}{2}(1+\frac s2)} + \left( \frac{|x-y|}{x_d\vee y_d}\right)^{2p-\alpha(1+\frac s2)}\one_{p>\frac{\alpha}{2}(1+\frac s2)} \right) \one_{\alpha<2} \Biggr].
      \end{split}
    \end{align}
  \end{enumerate}
\end{theorem}

\begin{remark}
	\label{rieszrem}
	Let $\alpha\in(0,2)$, $\lambda\in[\lambda_*,0)$ and assume that $\me{-tL_\lambda}(x,y)$ satisfies the bound in \eqref{eq:heatkernel} with $p$ defined by \eqref{eq:defp}. Then \eqref{eq:riesz} and \eqref{eq:rieszimprovedsim} remain valid. Similarly, the upper (resp.~lower) bound in \eqref{eq:heatkernel} implies the upper (resp.~lower) bound in \eqref{eq:riesz} and \eqref{eq:rieszimprovedsim}. This follows by the same arguments as in the proof below.
\end{remark}

Note that when $\alpha=2$ or when $\alpha<2$ and $p<\frac\alpha2(1+\frac s2)$ the bound in the theorem can be written as
\begin{equation}
	\label{eq:rieszexpected}
	L_\lambda^{-s/2}(x,y)  \sim_{d,\alpha,\lambda,s} |x-y|^{\alpha\frac{s}{2}-d} \left(1\wedge\frac{x_d}{|x-y|} \right)^{p} \left(1 \wedge\frac{y_d}{|x-y|}\right)^{p}
\end{equation}
for all $x,y\in\R^d_+$. This is reminiscent of the Riesz kernel bounds in \cite{Killipetal2016,Franketal2021}. Remarkably, a bound of this form does not hold globally when $\alpha<2$ and $p\geq\frac\alpha2 (1+\frac s2)$, and in the region $x_d \vee y_d \leq |x-y|$ the Riesz kernel is larger than the right side in \eqref{eq:rieszexpected}. This is a consequence of the slow off-diagonal decay of the heat kernel in the case $\alpha<2$. We will see in the following sections that this worse behavior does not lead to additional restrictions in the generalized Hardy inequality or the reverse Hardy inequality.

\begin{proof}
  By the spectral theorem, the Riesz kernel can be represented as
  \begin{align*}
    L_\lambda^{-s/2}(x,y) = \frac{1}{\Gamma(s/2)}\int_0^\infty\me{-tL_{\lambda}}(x,y)\,t^{s/2} \,\frac{dt}{t} \,.
  \end{align*}
  Inserting the two-sided bounds for $\me{-tL_{\lambda}}(x,y)$ in \eqref{eq:heatkernel} and \eqref{eq:localheatkernelhardy} and changing variables, we see that the left side of \eqref{eq:riesz} is comparable to
  \begin{align*}
    \begin{split}
      &\int_0^\infty \frac{dt}{t}\, t^{-\frac d\alpha+\frac{s}{2}}
      \left(1\wedge\frac{x_d}{t^{1/\alpha}}\right)^p
      \left(1\wedge\frac{y_d}{t^{1/\alpha}}\right)^p \left[\left(1\wedge\frac{t^{1+d/\alpha}}{|x-y|^{d+\alpha}}\right)\one_{\alpha<2}+\exp\left(-\frac{c|x-y|^2}{t}\right)\one_{\alpha=2}\right]\\
      & \quad = |x-y|^{\alpha\frac{s}{2}-d}
      \int_0^\infty \frac{d\tau}{\tau}\, \tau^{-1-\frac s2} \left[\left( 1 \wedge \tau^{\frac{d}{\alpha}+1} \right)\one_{\alpha<2} + \tau^{\frac d2+1}\me{-c\tau}\one_{\alpha=2}\right] \\
      & \qquad\qquad\qquad\qquad\qquad\qquad \times \left(1\wedge\frac{x_d\,\tau^{1/\alpha}}{|x-y|}\right)^p \left(1\wedge\frac{y_d\,\tau^{1/\alpha}}{|x-y|}\right)^p
    \end{split}
  \end{align*}
  for certain $c>0$, possibly different for the upper and lower bounds. The integral is similar to that in \cite[Lemma~5.2]{Killipetal2016} (or \cite[(2.3)]{Franketal2021}, but with $x_d$ and $y_d$ in place of $|x|$ and $|y|$ and $p$ in place of $-\delta$). There are, however, some differences, in particular in the case $\alpha<2$ and $p\geq\frac\alpha 2(1+\frac s2)$, so we include the details of the bounds.
  
  We shall show that for all $T,S>0$ with $|T^{-\frac1\alpha}-S^{-\frac1\alpha}|\leq 1$ we have
  \begin{align*}
    & \int_0^\infty \frac{d\tau}{\tau}\, \tau^{-1-\frac s2} \left[\left( 1 \wedge \tau^{\frac{d}{\alpha}+1} \right)\one_{\alpha<2} + \tau^{\frac d2+1}\me{-c\tau}\one_{\alpha=2}\right] 
    \left(1\wedge (\tau/T)^{1/\alpha} \right)^p \left(1\wedge (\tau/S)^{1/\alpha} \right)^p \\
    & \sim
    \begin{cases}
      \left( 1\wedge T^{-\frac1\alpha} \wedge S^{-\frac1\alpha} \right)^p & \text{if}\ T \wedge S\leq 1,\\
      \begin{aligned}
        & (TS)^{-\frac p\alpha} \Bigl[\one_{\alpha=2} \\
        & \ + \left(\one_{p\leq\frac{\alpha}{2}(1+\frac s2)} + \ln \left(T\wedge S\right)\one_{p=\frac{\alpha}{2}(1+\frac s2)} + (T\wedge S)^{\frac{2p}{\alpha}-1-\frac s2} \right) \one_{\alpha<2} \Bigr]
      \end{aligned} & \text{if}\ T\wedge S\geq 1.
    \end{cases}
  \end{align*}
  Setting $T:= (|x-y|/x_d)^\alpha$, $S:=(|x-y|/y_d)^\alpha$, we easily deduce from this the assertion. Note that the bound $|T^{-\frac1\alpha}-S^{-\frac1\alpha}|\leq 1$ comes from $|x_d-y_d|\leq|x-y|$.
  
  To prove the above assertion, by symmetry we may assume that $S\leq T$.
  
  \medskip
  \textit{Case $S\leq T\leq 1$}. In this case we have $S^{-\frac1\alpha}\leq T^{-\frac1\alpha}+1\leq 2 T^{-\frac1\alpha}$ and so $S\sim T$. Thus, the relevant integral is comparable to
  $$
  \int_0^\infty \frac{d\tau}{\tau}\, \tau^{-1-\frac s2} \left[\left( 1 \wedge \tau^{\frac{d}{\alpha}+1} \right)\one_{\alpha<2} + \tau^{\frac d2+1}\me{-c\tau}\one_{\alpha=2}\right] 
  \left(1\wedge (\tau/T)^{1/\alpha} \right)^{2p} \,,
  $$
  and we claim that this is comparable to $1$. Indeed,
  \begin{align*}
  	& \int_T^\infty \frac{d\tau}{\tau}\, \tau^{-1-\frac s2} \left[\left( 1 \wedge \tau^{\frac{d}{\alpha}+1} \right)\one_{\alpha<2} + \tau^{\frac d2+1}\me{-c\tau}\one_{\alpha=2}\right] 
  	\left(1\wedge (\tau/T)^{1/\alpha} \right)^{2p} \\
	& = \int_T^\infty \frac{d\tau}{\tau}\, \tau^{-1-\frac s2} \left[\left( 1 \wedge \tau^{\frac{d}{\alpha}+1} \right)\one_{\alpha<2} + \tau^{\frac d2+1}\me{-c\tau}\one_{\alpha=2}\right] \sim 1 \,,
  \end{align*}
  since $T\leq 1$ and since the integral converges at both zero (according to the assumption $\frac s2<\frac d\alpha$) and infinity. For a lower bound we drop the integral between $0$ and $T$ and for an upper bound, we estimate it by
  \begin{align*}
  	& \int_0^T \frac{d\tau}{\tau}\, \tau^{-1-\frac s2} \left[\left( 1 \wedge \tau^{\frac{d}{\alpha}+1} \right)\one_{\alpha<2} + \tau^{\frac d2+1}\me{-c\tau}\one_{\alpha=2}\right] 
  	\left(1\wedge (\tau/T)^{1/\alpha} \right)^{2p} \\
  	& \leq \int_0^T \frac{d\tau}{\tau}\, \tau^{-1-\frac s2} \tau^{\frac{d}{\alpha}+1} (\tau/T)^\frac{2p}{\alpha}
  	\sim T^{-\frac s2+\frac d\alpha} \leq 1 \,,
  \end{align*}
  since $\frac{\alpha s}{2}<d\wedge (d+2p)$ ensures the convergence of the integral and the last inequality.
  
  \medskip
  \textit{Case $S\leq 1\leq T$}. In this case we have $S^{-\frac1\alpha}\leq T^{-\frac1\alpha} + 1\leq 2$ and so $S\sim 1$. Thus, the relevant integral is comparable to
  $$
  \int_0^\infty \frac{d\tau}{\tau}\, \tau^{-1-\frac s2} \left[\left( 1 \wedge \tau^{\frac{d}{\alpha}+1} \right)\one_{\alpha<2} + \tau^{\frac d2+1}\me{-c\tau}\one_{\alpha=2}\right] 
  \left(1\wedge (\tau/T)^{1/\alpha} \right)^p \left(1\wedge \tau^{1/\alpha} \right)^p \,,
  $$
  and we claim that this is comparable to $T^{-\frac p\alpha}$. Indeed, we have, using $\frac{s}{2}<\frac{d+2p}{\alpha}$,
  \begin{align*}
  	& \int_0^1 \frac{d\tau}{\tau}\, \tau^{-1-\frac s2} \left[\left( 1 \wedge \tau^{\frac{d}{\alpha}+1} \right)\one_{\alpha<2} + \tau^{\frac d2+1}\me{-c\tau}\one_{\alpha=2}\right] 
  	\left(1\wedge (\tau/T)^{1/\alpha} \right)^p \left(1\wedge \tau^{1/\alpha} \right)^p \\
  	& \sim T^{-\frac p\alpha} \int_0^1 \frac{d\tau}{\tau}\, \tau^{-1-\frac s2} \tau^{\frac{d}{\alpha}+1} \tau^\frac{2p}\alpha \sim T^{-\frac p\alpha} \,.
  \end{align*} 
  For a lower bound we drop the integral between $1$ and $\infty$ and for an upper bound, we estimate, using $1+\frac{s}{2}-\frac{p}{\alpha}>0$ when $\alpha<2$ (as a consequence of $s>0$ and $p<\alpha$),
  \begin{align*}
  	& \int_1^T \frac{d\tau}{\tau}\, \tau^{-1-\frac s2} \left[\left( 1 \wedge \tau^{\frac{d}{\alpha}+1} \right)\one_{\alpha<2} + \tau^{\frac d2+1}\me{-c\tau}\one_{\alpha=2}\right] 
  	\left(1\wedge (\tau/T)^{1/\alpha} \right)^p \left(1\wedge \tau^{1/\alpha} \right)^p \\
  	& \sim \int_1^T \frac{d\tau}{\tau}\, \tau^{-1-\frac s2} \left[ \one_{\alpha<2} + \tau^{\frac d2+1}\me{-c\tau}\one_{\alpha=2}\right] (\tau/T)^{\frac p\alpha} \lesssim T^{-\frac p\alpha}
  \end{align*}
  and, using again $1+\frac{s}{2}-\frac{p}{\alpha}\geq 0$ when $\alpha<2$,
  \begin{align*}
  	& \int_T^\infty \frac{d\tau}{\tau}\, \tau^{-1-\frac s2} \left[\left( 1 \wedge \tau^{\frac{d}{\alpha}+1} \right)\one_{\alpha<2} + \tau^{\frac d2+1}\me{-c\tau}\one_{\alpha=2}\right] 
  	\left(1\wedge (\tau/T)^{1/\alpha} \right)^p \left(1\wedge \tau^{1/\alpha} \right)^p \\
  	& = \int_T^\infty \frac{d\tau}{\tau}\, \tau^{-1-\frac s2} \left[ \one_{\alpha<2} + \tau^{\frac d2+1}\me{-c\tau}\one_{\alpha=2}\right] 
  	\sim T^{-1-\frac s2} \one_{\alpha<2} + T^{-1-\frac s2 + \frac d2} e^{-cT} \one_{\alpha=2} \lesssim T^{-\frac p\alpha} \,.
  \end{align*}
  
  \medskip
  \textit{Case $1\leq S\leq T$}. We split the relevant integral into three pieces, by cutting at $S$ and at $T$. For the first integral we find, using $s<\frac{2(d+2p)}\alpha$,
  \begin{align*}
  	& \int_0^S \frac{d\tau}{\tau}\, \tau^{-1-\frac s2} \left[\left( 1 \wedge \tau^{\frac{d}{\alpha}+1} \right)\one_{\alpha<2} + \tau^{\frac d2+1}\me{-c\tau}\one_{\alpha=2}\right] 
  	\left(1\wedge (\tau/T)^{1/\alpha} \right)^p \left(1\wedge (\tau/S)^{1/\alpha} \right)^p \\
  	& = (ST)^{-\frac p\alpha} \int_0^S \frac{d\tau}{\tau}\, \tau^{-1-\frac s2} \left[\left( 1 \wedge \tau^{\frac{d}{\alpha}+1} \right)\one_{\alpha<2} + \tau^{\frac d2+1}\me{-c\tau}\one_{\alpha=2}\right] \tau^{\frac{2p}\alpha} \\
  	& \sim (ST)^{-\frac p\alpha} \left( \left( \one_{p\leq\frac\alpha2(1+\frac s2)} + (\ln S) \one_{p=\frac\alpha 2(1+\frac s2)} + S^{\frac{2p}\alpha - 1 - \frac s2} \one_{p>\frac\alpha 2(1+\frac s2)} \right) \one_{\alpha<2} + \one_{\alpha=2} \right).
  \end{align*}
  This term is of the claimed form. Thus, for a lower bound we can drop the integral between $S$ and $\infty$.
  
  We bound the second integral from above by
  \begin{align*}
  	& \int_S^T \frac{d\tau}{\tau}\, \tau^{-1-\frac s2} \left[\left( 1 \wedge \tau^{\frac{d}{\alpha}+1} \right)\one_{\alpha<2} + \tau^{\frac d2+1}\me{-c\tau}\one_{\alpha=2}\right] 
  	\left(1\wedge (\tau/T)^{1/\alpha} \right)^p \left(1\wedge (\tau/S)^{1/\alpha} \right)^p \\
  	& = T^{-\frac p\alpha} \int_S^T \frac{d\tau}{\tau}\, \tau^{-1-\frac s2} \left[ \one_{\alpha<2} + \tau^{\frac d2+1}\me{-c\tau}\one_{\alpha=2}\right]  \tau^\frac p\alpha \\
  	& \lesssim T^{-\frac p\alpha} \left( S^{-1-\frac s2+\frac p\alpha} \one_{\alpha<2} + S^{-1-\frac s2+\frac p\alpha +\frac d2} e^{-cS} \one_{\alpha=2} \right).
  \end{align*}
  When $\alpha<2$ and $p>\frac\alpha2(1+\frac s2)$, this upper bound equals the size of the first integral, and for $p\leq \frac\alpha2(1+\frac s2)$ we bound $T^{-\frac p\alpha} S^{-1-\frac s2+\frac p\alpha}\leq (TS)^{-\frac p\alpha}$. When $\alpha=2$, have clearly $T^{-\frac p\alpha} S^{-1-\frac s2+\frac p\alpha +\frac d2} e^{-cS}\lesssim (TS)^{-\frac p\alpha}$.
  
  We bound the third integral exactly as in the case $S\leq 1\leq T$ and obtain
  \begin{align*}
  	& \int_T^\infty \frac{d\tau}{\tau}\, \tau^{-1-\frac s2} \left[\left( 1 \wedge \tau^{\frac{d}{\alpha}+1} \right)\one_{\alpha<2} + \tau^{\frac d2+1}\me{-c\tau}\one_{\alpha=2}\right] 
  	\left(1\wedge (\tau/T)^{1/\alpha} \right)^p \left(1\wedge (\tau/S)^{1/\alpha} \right)^p \\
  	& = \int_T^\infty \frac{d\tau}{\tau}\, \tau^{-1-\frac s2} \left[ \one_{\alpha<2} + \tau^{\frac d2+1}\me{-c\tau}\one_{\alpha=2}\right] 
  	\sim T^{-1-\frac s2} \one_{\alpha<2} + T^{-1-\frac s2 + \frac d2} e^{-cT} \one_{\alpha=2}  \,.
  \end{align*}
  When $\alpha<2$ and $p\leq\frac\alpha2(1+\frac s2)$, we bound
  $$
  T^{-1-\frac s2}\leq (ST)^{-\frac12(1+\frac s2)}
  \leq (ST)^{-\frac p\alpha} \left( \one_{p\leq\frac\alpha2(1+\frac s2)} + (\ln S) \one_{p=\frac\alpha 2(1+\frac s2)} \right)
  $$
  and when $\alpha<2$ and $p>\frac\alpha2(1+\frac s2)$, we bound, recalling $\frac{p}{\alpha}<1+\frac s2$,
  $$
  T^{-1-\frac s2}\leq (ST)^{-\frac p\alpha} S^{\frac{2p}\alpha - 1-\frac s2} \,.
  $$
  When $\alpha=2$, we have $p\geq\frac12>0$ and therefore $T^{-1-\frac s2 + \frac d2} e^{-cT}\lesssim (ST)^{-\frac p\alpha}$. This completes the proof.
\end{proof}

\section{Proof of the generalized Hardy inequality (Theorem \ref{genhardy})}
\label{s:genhardy}

We first prove a theorem that is closely related to Theorem \ref{genhardy}.

\begin{theorem}
	\label{genhardybdd}
	Let
	$\alpha\in(0,2]$ and let
	$\lambda\geq0$ when $\alpha\in(0,2)$ and $\lambda\geq-1/4$ when $\alpha=2$.
	Let $p$ be defined by \eqref{eq:defp}. Then, if $s\in(0,\frac{1+2p}{\alpha} \wedge \frac{2d}\alpha)$, one has
	\begin{align}
		\label{eq:genhardybdd}
		\| x_d^{-\alpha s/2} L_\lambda^{-s/2} g\|_{L^2(\R_+^d)}
		\lesssim_{d,\alpha,\lambda,s} \| g \|_{L^2(\R_+^d)}
		\qquad\text{for all}\ g\in L^2(\R_+^d) \,.
	\end{align}
	Conversely, if \eqref{eq:genhardybdd} holds for some $s\in(0,\frac{2d}{\alpha}\wedge\frac{2(d+2p)}{\alpha})$, then $s<\frac{1+2p}\alpha$.
\end{theorem}

\begin{remark}\label{genhardybddrem} 
	Let $\alpha\in(0,2)$, $\lambda\in[\lambda_*,0)$ and assume that $\me{-tL_\lambda}(x,y)$ satisfies the bound in \eqref{eq:heatkernel} with $p$ defined by \eqref{eq:defp}. Then the assertions of Theorem \ref{genhardybdd} remain valid. Similarly, the upper (resp.\ lower) bound in \eqref{eq:heatkernel} implies the sufficiency (resp.\ necessity) of the assumption $s<\frac{1+2p}{\alpha}$ for the validity of \eqref{eq:genhardybdd}. This follows by the same arguments as in the proof below, taking into account Remark \ref{rieszrem}.
\end{remark}

The basic strategy of the proof is to use Theorem \ref{riesz} in order to replace the operator $L_{\lambda}^{- s/2}$ by one with a more explicit kernel.

\begin{proof}
	We assume throughout that $\alpha$, $\lambda$ and $p$ are as in the statement of the theorem and that $s\in(0,\frac{2d}{\alpha}\wedge\frac{2(d+2p)}{\alpha})$.
	
	\medskip
	
	\emph{Necessity of the assumption $s<\frac{1+2p}{\alpha}$.} 
	We consider a similar example as in \cite[p.~1283]{Killipetal2018}. Let $w=(0,0,...,2)\in\R_+^d$ and $0\leq\phi\in C_c^\infty(\R^d_+)$ with $\phi\geq 1$ in $B_{1/2}(w)$. We use part (a) of Theorem~\ref{riesz}. (More precisely, we also use part (b) to see that the bound in part (a) is also valid for $x_d\vee y_d<|x-y|\leq 2(x_d\vee y_d)$.) This shows that, for $x\in\R_+^d$ with $|x|\leq1$, we have
	\begin{align*}
          \begin{split}
            (L_\lambda^{-s/2}\phi)(x)
            & \geq \int_{\R_+^d}dy\, \one_{|x-y| \leq 2 y_d} L_{\lambda}^{-s/2}(x,y)\phi(y)\\
            & \gtrsim \int_{\R_+^d}dy\, \one_{|x-y| \leq 2 y_d} \one_{|y-w|\leq\frac12} |x-y|^{\alpha\frac{s}{2}-d}\left(1\wedge\frac{x_d}{|x-y|} \wedge\frac{y_d}{|x-y|}\right)^{p}\\
            & \gtrsim x_d^p \int_{\R_+^d}dy\, \one_{|x-y| \leq 2 y_d}\one_{|y-w|\leq\frac12}
            \gtrsim x_d^p\,. 
          \end{split}
	\end{align*}
	In the third inequality we used the fact that $|x-y|\sim 1$ and $1\wedge\frac{x_d}{|x-y|} \wedge\frac{y_d}{|x-y|} \sim \frac{x_d}{|x-y|}$ on the domain of integration. (Indeed, clearly, $\frac12\leq |x-y|\leq |x|+|y-w|+|w| \leq \frac 72$, $x_d\leq 1$ and $y_d\geq\frac32$.) In the fourth inequality we used the fact that the inequality $|x-y|\leq 2y_d$ is satisfied for all $y$ with $|y-w|\leq\frac12$. (Indeed, $|x-y|\leq |x'-y'|+(y_d-x_d)\leq |x'|+|y'|+y_d$, where $|x'|\leq |x|\leq1\leq \frac23y_d$ and $|y'|\leq|y-w|\leq\frac12\leq\frac13 y_d$.)
	
	This allows us to bound
	\begin{align*}
		\|x_d^{-\frac{\alpha s}{2}}L_\lambda^{-s/2}\phi\|_{L^2(\R_+^d)}
		\geq \|\one_{|x|\leq1} x_d^{-\frac{\alpha s}{2}}L_\lambda^{-s/2}\phi\|_{L^2(\R_+^d)}
		\gtrsim \|\one_{|x|\leq1} x_d^{p-\frac{\alpha s}{2}}\|_{L^2(\R_+^d)} \,.
	\end{align*}
	Since the right side is infinite if $s\geq\frac{1+2p}{\alpha}$, we see that the inequality $s<\frac{1+2p}{\alpha}$ is necessary for the validity of \eqref{eq:genhardy}.
	
	\medskip
	
	\emph{Sufficiency of the assumption $s<\frac{1+2p}{\alpha}$.} 
	We shall prove the $L^2(\R_+^d)$-boundedness of the operator with kernel $x_d^{-\frac{\alpha s}{2}} L_\lambda^{-\frac s2}(x,y)$. By the upper bounds in Theorem \ref{riesz}, it suffices to prove the $L^2(\R_+^d)$-boundedness of the operator with kernel $K(x,y)$, defined to be $x_d^{-\frac{\alpha s}{2}}$ times the function appearing in the bounds in Theorem \ref{riesz}. We will divide $K$ into four pieces supported in essentially disjoint sets and show boundedness of the resulting four operators. To that end we perform Schur tests as in \cite[Proposition~3.2]{Killipetal2018} (with $s$ in place of $\frac{\alpha s}2$ and $\sigma$ in place of $-p$). These Schur tests involve weights and the weights are chosen differently for the four different pieces of $K$.
	
	The four regions are defined by $|x-y|\leq 4(x_d\wedge y_d)$, $4x_d\leq |x-y|\leq 4y_d$, $4y_d\leq|x-y|\leq 4x_d$ and $4(x_d\vee y_d)\leq |x-y|$. The factors of $4$ will be convenient in some regions and we observe that Theorem~\ref{riesz} as stated is equivalent to a variant of Theorem~\ref{riesz} where the distinction between parts (a) and (b) includes similar factors of $4$.
	
	\medskip
	\textbf{Case $|x-y|\leq 4(x_d\wedge y_d)$.}
	In this case we have $1\wedge\frac{x_d}{|x-y|}\wedge\frac{y_d}{|x-y|}\sim 1$ and therefore the kernel becomes
	\begin{align*}
		K(x,y) \sim x_d^{-\alpha\frac{s}{2}} |x-y|^{\alpha\frac{s}{2}-d}\,.
	\end{align*}
	For the first half of the Schur test we bound
	\begin{align*}
		\int\limits_{|x-y| \leq 4(x_d\wedge y_d)}x_d^{-\alpha\frac{s}{2}} |x-y|^{\alpha\frac{s}{2}-d}\,dy
		\leq \int\limits_{|x-y| \leq 4x_d}x_d^{-\alpha\frac{s}{2}} |x-y|^{\alpha\frac{s}{2}-d}\,dy
        \lesssim 1 \,.
	\end{align*}
	For the second half of the Schur test, we note that $y_d \leq x_d + |x-y| \leq x_d + 4(x_d\wedge y_d)\leq 5 x_d$ and therefore $x_d^{-\alpha\frac s2}$ in the kernel can be replaced by $y_d^{-\alpha\frac s2}$. Therefore, the second half of the Schur test is similar to the first, and we deduce the $L^2(\R_+^d)$-boundedness of the piece of $K$ in this region.
	
	\medskip
	\textbf{Case $4x_d \leq |x-y| \leq 4y_d$.}
	In this case, we have $1\wedge\frac{x_d}{|x-y|}\wedge\frac{y_d}{|x-y|} \sim \frac{x_d}{|x-y|}$ and therefore the kernel becomes
	\begin{align*}
		K(x,y) = x_d^{p-\frac{\alpha s}2}|x-y|^{\frac{\alpha s}2-d-p} \,.
	\end{align*}
	We perform a Schur test with weight
	\begin{align*}
		w(x,y) = \left(\frac{x_d}{|x-y|}\right)^\beta
		\quad \text{with}\ \frac{\alpha s}{2} - p < \beta < 1+p-\frac{\alpha s}{2} \,.
	\end{align*}
	The assumption $s<\frac{1+2p}\alpha$ guarantees that one can find such a $\beta$. 
	
	For the first half of the Schur test we bound
	\begin{align*}
		\begin{split}
			& \int\limits_{4x_d\leq|x-y|\leq4y_d}w(x,y) \, K(x,y)\,dy
			\lesssim x_d^{-\frac{\alpha s}{2}+p+\beta}\int\limits_{|x-y|\geq x_d/4}|x-y|^{\frac{\alpha s}{2}-d-p-\beta}\,dy
			\lesssim 1 \,,
		\end{split}
	\end{align*}
	where the finiteness of the integral comes from the choice of $\beta$. For the second half of the Schur test we note that in our region we have $|x-y|\geq y_d-x_d \geq y_d-|x-y|/4$, so $|x-y|\geq 4y_d/5$. We bound
	\begin{align*}
          \begin{split}
            & \int\limits_{4x_d\leq|x-y|\leq4y_d}w(x,y)^{-1} \, K(x,y) \,dx\\
            & \quad \lesssim \int\limits_{4y_d/5\leq |x-y| \leq 4y_d} x_d^{-\frac{\alpha s}2+ p-\beta}|x-y|^{\frac{\alpha s}2 -d-p+\beta} \,dx \\
            & \quad = \int\limits_{4/5 \leq |w| \leq 4}(w_d+1)^{-\frac{\alpha s}2+p-\beta} |w|^{\frac{\alpha s}2-d-p+\beta}\one_{\{w_d>-1\}}\,dw <\infty \,,
		\end{split}
	\end{align*}
	where we changed variables $x-y = y_d w$ and where the finiteness of the integral comes from the choice of $\beta$. We deduce the $L^2(\R_+^d)$-boundedness of the piece of $K$ in this region.
	
	\medskip
	\textbf{Case $4y_d \leq |x-y|\leq 4x_d$.}
	In this case, we have $1\wedge\frac{x_d}{|x-y|}\wedge\frac{y_d}{|x-y|} \sim \frac{y_d}{|x-y|}$ and therefore the kernel becomes
	\begin{align*}
		K(x,y) = x_d^{-\frac{\alpha s}2} y_d^p |x-y|^{\frac{\alpha s}2-d-p} \,.
	\end{align*}
	We perform a Schur test with weight
	\begin{align*}
		w(x,y) = \left(\frac{|x-y|}{y_d}\right)^\gamma
		\quad \text{with}\ - p < \gamma < 1+p \,.
	\end{align*}
	Since $p \geq \frac{\alpha-1}2>-\frac12$ it is possible to find such a $\gamma$.	
	
	Similarly as in the previous case (but with $x$ and $y$ interchanged), we have $|x-y|\geq 4x_d/5$ and, in particular, $y_d\leq x_d\sim|x-y|$. Therefore, for the first half of the Schur test we bound
	\begin{align*}
          \begin{split}
            \int\limits_{4y_d\leq|x-y|\leq4x_d} w(x,y) \, K(x,y) \,dy
            & \lesssim x_d^{-d-p+\gamma} \int\limits_{4y_d\leq|x-y|\leq4x_d} y_d^{p-\gamma} \,dy\\
            & \leq x_d^{-d-p+\gamma} \int\limits_{|x'-y'|\leq 4x_d \,,\ y_d\leq x_d} y_d^{p-\gamma} \,dy\\
            & = \int\limits_{|w'|\leq 4 \,,w_d\leq 1} w_d^{p-\gamma} \,dw <\infty \,,
		\end{split}
	\end{align*}
	where the finiteness of the integral comes for the choice of $\gamma$. For the second half of the Schur test we bound, using again $x_d\sim |x-y|$,
	\begin{align*}
		\int\limits_{4y_d\leq|x-y|\leq4x_d} w(x,y)^{-1} \, K(x,y) \,dx
		& \lesssim y_d^{p+\gamma} \int\limits_{|x-y|\geq4y_d}|x-y|^{-d-p-\gamma}\,dx \\
		& = \int\limits_{|w|\geq4}|w|^{-d-p-\gamma}\,dw <\infty \,,
	\end{align*}
	where the finiteness of the integral comes from the choice of $\gamma$. We deduce the $L^2(\R_+^d)$-boundedness of the piece of $K$ in this region.
	
	\medskip
	\textbf{Case $4(x_d \vee y_d) \leq |x-y|$.}
	In this region the kernel is
	\begin{align*}
          \begin{split}
            & K(x,y)
            = x_d^{-\alpha\frac{s}{2}} |x-y|^{\alpha\frac{s}{2}-d} \left(\frac{x_d\,y_d}{|x-y|^2}\right)^p \cdot \Biggl[ \one_{\alpha=2} \\
            & \ + \left(\one_{p\leq\frac{\alpha}{2}(1+\frac s2)} + \left(\ln\frac{|x-y|}{x_d\vee y_d}\right)\one_{p=\frac{\alpha}{2}(1+\frac s2)} + \left( \frac{|x-y|}{x_d\vee y_d}\right)^{2p-\alpha(1+\frac s2)} \one_{p>\frac{\alpha}{2}(1+\frac s2)} \right) \one_{\alpha<2} \Biggr].
          \end{split}
	\end{align*}
	We perform a Schur test with weight
	\begin{align*}
		w(x,y) = \left(\frac{x_d}{|x-y|}\right)^\beta\left(\frac{|x-y|}{y_d}\right)^\gamma
		\ \text{with}\ \tfrac{\alpha s}2-p<\beta<1+p-\tfrac{\alpha s}{2} \,,\ -p<\gamma<1+p \,.
	\end{align*}
	When $\alpha<2$ and $p>\frac\alpha2(1+\frac s2)$, we also assume that
	$$
	-1-\alpha < \beta - \gamma< 1+\alpha \,.
	$$
	A possible parameter choice that satisfies all the constraints is $\beta=\gamma=\frac12$.

	For the first Schur test we bound
	\begin{align*}
		& \int_{4(x_d\vee y_d)\leq |x-y|} w(x,y) K(x,y)\,dy = \sum_{R\in 2^\Z} \int_{4(x_d\vee y_d)\leq |x-y|} \one_{R\leq |x-y|< 2R} w(x,y) K(x,y)\,dy \\
		& \lesssim \sum_{2x_d< R\in 2^\Z} \int_{\R^d_+} \one_{|x'-y'|< 2R} \one_{2y_d< R} w_R(x,y) K_R(x,y)\,dy \,,
	\end{align*}
	where $w_R$ and $K_R$ are defined as $w$ and $K$, but with $|x-y|$ at each occurrence replaced by $R$. For fixed $R\in 2^\Z$ with $R>2x_d$, we carry out the $y'$-integration and, if $\alpha<2$ and $p=\frac\alpha2(1+\frac s2)$, we bound $x_d\vee y_d\geq y_d$. In this way, we obtain
	\begin{align*}
          & \int_{\R^d_+} \one_{|x'-y'|< 2R} \one_{2y_d< R} w_R(x,y) K_R(x,y)\,dy
            \lesssim x_d^{-\frac{\alpha s}{2}+p+\beta}  R^{\frac{\alpha s}{2}-2p-\beta+\gamma-1} \\
          & \times \int_0^{R/2} dy_d\, y_d^{p-\gamma} \cdot \Biggl[ \one_{\alpha=2} \\
          & \qquad\qquad + \left(\one_{p\leq\frac{\alpha}{2}(1+\frac s2)} + \left( \ln \frac{R}{y_d}\right)\one_{p=\frac{\alpha}{2}(1+\frac s2)} + \left( \frac{R}{x_d\vee y_d}\right)^{2p-\alpha(1+\frac s2)} \one_{p>\frac{\alpha}{2}(1+\frac s2)} \right) \one_{\alpha<2} \Biggr] \\
          & \lesssim x_d^{-\frac{\alpha s}{2}+p+\beta}  R^{\frac{\alpha s}{2}-p-\beta}
            + x_d^{\alpha+\beta-\gamma+1} R^{-\alpha-\beta+\gamma-1} \one_{p>\frac\alpha2(1+\frac s2)} \one_{\alpha<2} \,.
	\end{align*}
	Here the assumption $p-\gamma>-1$ guarantees the $y_d$-integral to converge near zero. The additional term in case $\alpha<2$ and $p>\frac\alpha2(1+\frac s2)$ comes from the integral between $0$ and $x_d$.
	
	Summing with respect to $R$ we obtain
	$$
	\sum_{2x_d< R\in 2^\Z} \left( x_d^{-\frac{\alpha s}{2}+p+\beta}  R^{\frac{\alpha s}{2}-p-\beta}
	+ x_d^{\alpha+\beta-\gamma+1} R^{-\alpha-\beta+\gamma-1} \one_{p>\frac\alpha2(1+\frac s2)} \one_{\alpha<2} \right) \sim 1 \,.
	$$
	Here the assumptions $\beta>\frac{\alpha s}{2}-p$ and $\beta-\gamma>-1-\alpha$ guarantee the convergence of the $R$-sum.
	
	For the second Schur test we bound similarly
	\begin{align*}
          & \int_{4(x_d\vee y_d)\leq |x-y|} w(x,y)^{-1} K(x,y)\,dx \\
          & \quad = \sum_{R\in 2^\Z} \int_{4(x_d\vee y_d)\leq |x-y|} \one_{R\leq |x-y|< 2R} w(x,y)^{-1} K(x,y)\,dx \\
          & \quad \lesssim \sum_{2y_d< R\in 2^\Z} \int_{\R^d_+} \one_{|x'-y'|< 2R} \one_{2x_d< R} w_R(x,y)^{-1} K_R(x,y)\,dx \,.
	\end{align*}
	For fixed $R\in 2^\Z$ with $R>2y_d$, we carry out the $x'$-integration and, if $p =\frac\alpha2(1+\frac s2)$, we bound $x_d\vee y_d\geq x_d$. In this way, we obtain
	\begin{align*}
          & \int_{\R^d_+} \one_{|x'-y'|< 2R} \one_{2x_d< R} w_R(x,y)^{-1} K_R(x,y)\,dx
            \lesssim y_d^{p+\gamma} R^{\frac{\alpha s}{2}-2p+\beta-\gamma-1} \\
          & \times \int_0^{R/2} dx_d\, x_d^{-\frac{\alpha s}{2}+p-\beta} \cdot \Biggl[ \one_{\alpha=2} \\
          & \qquad\qquad + \left(\one_{p\leq\frac{\alpha}{2}(1+\frac s2)} + \left( \ln \frac{R}{x_d}\right)\one_{p=\frac{\alpha}{2}(1+\frac s2)} + \left( \frac{R}{x_d\vee y_d}\right)^{2p-\alpha(1+\frac s2)} \one_{p>\frac{\alpha}{2}(1+\frac s2)} \right) \one_{\alpha<2} \Biggr] \\
          & \lesssim y_d^{p+\gamma} R^{-p-\gamma}  + y_d^{\alpha-\beta+\gamma+1} R^{-\alpha+\beta-\gamma-1} \one_{p>\frac\alpha2(1+\frac s2)}\one_{\alpha<2} \,.
	\end{align*}
	Here the assumption $\beta<1+p-\frac{\alpha s}{2}$ guarantees the convergence of the $x_d$-integral near zero. The additional term in case $\alpha<2$ and $p>\frac{\alpha}{2} (1+\frac s2)$ comes from the integral between $0$ and $y_d$.
	
	Summing with respect to $R$ we obtain
	$$
	\sum_{2y_d<R\in 2^\Z} \left( y_d^{p+\gamma} R^{-p-\gamma}  + y_d^{\alpha-\beta+\gamma+1} R^{-\alpha+\beta-\gamma-1} \one_{p>\frac\alpha2(1+\frac s2)}\one_{\alpha<2} \right) \sim 1 \,.
	$$
	Here the assumptions $\gamma>-p$ and $\beta-\gamma<1+\alpha$ guarantee the convergence of the $R$-sum. This concludes the Schur test and we deduce the $L^2(\R^d_+)$ boundedness of the piece of $K$ in this last region.
\end{proof}

To deduce Theorem \ref{genhardy} from Theorem \ref{genhardybdd} we need the following lemma.

\begin{lemma}
  \label{domain}
  Let $\alpha,s\in(0,2]$ and $\lambda\geq\lambda_*$. Then $C_c^\infty(\R_+^d) \subset \dom L_\lambda^{s/2}$.
\end{lemma}

\begin{proof}
	Since the domains are nested as $s$ decreases, it suffices to consider the case $s=2$. The case $\alpha=2$ is classical, so we may assume $\alpha<2$. Let $f\in C^\infty_c(\R^d_+)$. By definition of the Friedrichs extension, we need to find a $g\in L^2(\R^d_+)$ such that
	\begin{align*}
		& \tfrac12\, \mathcal A(d,-\alpha) \iint_{\R_+^d\times\R_+^d} \frac{(\overline{u(x)}-\overline{u(y)})(f(x)-f(y))}{|x-y|^{d+\alpha}}\,dx\,dy + \lambda \int_{\R^d_+} \frac{\overline{u(x)}f(x)}{x_d^\alpha}\,dx \\
		& = \int_{\R^d_+} \overline{u(x)} g(x)\,dx
	\end{align*}
	for all $u\in C^1_c(\R^d_+)$. By polarizing the computation in Remark \ref{otherdef}, identifying both $f$ and $u$ with their extension by zero to $\R^d$, we see that this is equivalent to
	$$
	\int_{\R^d} \overline{(-\Delta)^{\alpha/4} u(x)} (-\Delta)^{\alpha/4} f(x)\,dx + (\lambda-\lambda_0) \int_{\R^d_+} \frac{\overline{u(x)}f(x)}{x_d^\alpha}\,dx \\
	= \int_{\R^d_+} \overline{u(x)} g(x)\,dx \,.
	$$
	This holds with $g :=((-\Delta)^{\alpha/2}f)|_{\R^d_+} + (\lambda-\lambda_0) x_d^{-\alpha}f$. Indeed, the first term belongs to $L^2(\R^d_+)$ since $|\xi|^\alpha \widehat f \in L^2(\R^d)$ and the second one since $x_d^{-\alpha}$ is bounded on the support of $f$. This completes the proof.
\end{proof}

\begin{proof}[Proof of Theorem \ref{genhardy}]
  For given $f\in C^\infty_c(\R^d_+)$, $g:=L_\lambda^{s/2} f\in L^2(\R^d_+)$ by Lemma \ref{domain}, so Theorem \ref{genhardy} follows from Theorem \ref{genhardybdd}.
\end{proof}

\begin{remark}
	The same proof, without invoking Lemma \ref{domain}, shows that the generalized Hardy inequality \eqref{eq:genhardy} holds for all $f\in\dom L_\lambda^{s/2}$ under the assumptions of Theorem \ref{genhardy}.
\end{remark}


\section{Difference of heat kernels}
\label{s:differenceheatkernels}

A key tool for the proof of the reversed Hardy inequality (Theorem~\ref{reversehardy}) are bounds for the difference between the heat kernels of $L_0$ and $L_\lambda$, i.e.,
\begin{align*}
  K_t^\alpha(x,y) := \me{-tL_0}(x,y) - \me{-tL_\lambda}(x,y) \,.
\end{align*}
Given $\alpha\in(0,2]$ and $\lambda\geq\lambda_*$, let $p$ be defined by \eqref{eq:defp} and set
$$
q:= \min\{ p, (\alpha-1)_+\} \,.
$$
We formulate our bounds in terms of the functions
\begin{align*}
  \begin{split}
    J_t^{\alpha} (x,y) & := \left( \one_{x_d \vee y_d\leq t^{1/\alpha}} + \one_{x_d \vee y_d\geq t^{1/\alpha}} \one_{|x-y|\geq (x_d\wedge y_d)/2} \right) 
    \left(1\wedge \frac{x_d}{t^{1/\alpha}}\right)^q \,\left(1\wedge \frac{y_d}{t^{1/\alpha}}\right)^q \\
    & \qquad \times t^{-\frac{d}{\alpha}} \left[\left( 1 \wedge \frac{t^{1+\frac{d}{\alpha}}}{|x-y|^{d+\alpha}} \right)\one_{\alpha<2} + \exp\left(-c\frac{|x-y|^2}{t}\right)\one_{\alpha=2}\right]
  \end{split}
\end{align*}
and, with some appropriate constant $c>0$,
\begin{align*}
  \begin{split}
    M_t^\alpha(x,y)
    & := \one_{x_d \vee y_d \geq t^{1/\alpha}} \one_{|x-y|\leq(x_d\wedge y_d)/2} \\
    & \quad \times \frac{t^{1-\frac{d}{\alpha}}}{(x_d\vee y_d)^\alpha} \left[\left( 1 \wedge \frac{t^{1+\frac{d}{\alpha}}}{|x-y|^{d+\alpha}} \right)\one_{\alpha<2} + \exp\left(-c\frac{|x-y|^2}{t}\right)\one_{\alpha=2}\right].
  \end{split}
\end{align*}

\begin{theorem}
  \label{differenceheatkernel3}
  Let
  $\alpha\in(0,2]$ and let
  $\lambda\geq0$ when $\alpha\in(0,2)$ and $\lambda\geq-1/4$ when $\alpha=2$.
  Then, for all $x,y\in\R_+^d$ and $t>0$, one has
  \begin{align}
    \label{eq:differenceheatkernel3}
    |K_t^\alpha(x,y)| \lesssim J_t^\alpha(x,y) + M_t^\alpha(x,y) \,.
  \end{align}
\end{theorem}

\begin{remark}
  \label{differenceheatkernel3rem}
  Let $\alpha\in(0,2)$, $\lambda\in[\lambda_*,0)$ and assume that $\me{-tL_\lambda}(x,y)$ satisfies the upper bound in \eqref{eq:heatkernel} with $p$ defined by \eqref{eq:defp}. Then \eqref{eq:differenceheatkernel3} remains valid. This follows by the same arguments as in the proof below.
\end{remark}

\begin{proof}
  We assume $\lambda\neq0$ without loss generality as the claim is trivial when $\lambda=0$. By scaling, it suffices to consider $t=1$ and, by symmetry, it suffices to consider $x_d \leq y_d$. We now drop the subscript $t$ in $K_t^\alpha$, $J_t^\alpha$, and $M_t^\alpha$.


  By the triangle inequality and the bounds \eqref{eq:heatkernel} and \eqref{eq:localheatkernelhardy}, we obtain
  \begin{align*}
    \begin{split}
      |K^\alpha(x,y)|
      & \lesssim 
      \left[ \left(1\wedge x_d \right)^p\,\left(1\wedge y_d \right)^p + \left(1\wedge x_d \right)^{(\alpha-1)_+} \,\left(1\wedge y_d \right)^{(\alpha-1)_+} \right]\\
      & \quad \times \left[\left(1\wedge |x-y|^{-d-\alpha}\right)\one_{\alpha<2} + \me{-c|x-y|^2}\one_{\alpha=2}\right].
    \end{split}
  \end{align*}
	For an upper bound we can replace both exponents $p$ and $(\alpha-1)_+$ by $q$ and arrive at the claimed bound in the regions where $y_d\leq 1$ and where $y_d\geq 1$ and $|x-y|\geq x_d/2$.
	
	\medskip
	
	In the following we concentrate on the region where $y_d\geq 1$ and $|x-y|\leq x_d/2$. Note that in this region we have $y_d\leq x_d + |x-y|\leq (3/2)x_d$, so $x_d\sim y_d\geq 1$.
	
	By Duhamel's formula, i.e.,
  \begin{align*}
    \me{-L_0} - \me{-L_\lambda}
    = \lambda \int_0^1 ds\, \me{-(1-s)L_0} x_d^{-\alpha}\me{-sL_{\lambda}}\,,
  \end{align*}
  and the bounds \eqref{eq:heatkernel} and \eqref{eq:localheatkernelhardy}, we conclude
  \begin{align*}
    \begin{split}
      |K^\alpha(x,y)|
      & \lesssim \int_0^1 ds \int_{\R_+^d} \frac{dz}{z_d^\alpha} s^{-\frac{d}{\alpha}}(1-s)^{-\frac{d}{\alpha}} \left(1\wedge\frac{z_d}{(1-s)^{1/\alpha}}\right)^{(\alpha-1)_+} \left(1\wedge\frac{z_d}{s^{1/\alpha}}\right)^{p} \\
      & \quad \times 
      \left[\left(1\wedge\frac{(1-s)^{1+\frac{d}{\alpha}}}{|x-z|^{d+\alpha}}\right)
        \left(1\wedge\frac{s^{1+\frac{d}{\alpha}}}{|y-z|^{d+\alpha}}\right)\one_{\alpha<2}\right.\\
      & \qquad\qquad\qquad\qquad\qquad\qquad\quad \left. + \exp\left(-c\left(\frac{|x-z|^2}{1-s}+\frac{|y-z|^2}{s}\right)\right)\one_{\alpha=2}\right]\,.
    \end{split}
  \end{align*}
	Note that here we dropped the factors
	\begin{equation}
		\label{eq:trivialfactors}
		\left(1\wedge\frac{x_d}{(1-s)^{1/\alpha}}\right)^{(\alpha-1)_+} \left(1\wedge\frac{y_d}{s^{1/\alpha}}\right)^{p} \sim 1 \,,
	\end{equation}
	since $x_d \sim y_d \geq 1$ and $s\in[0,1]$.
	
	We divide the $z$ integration at $z_d=x_d/2$, leading to the bound
	$$
	|K^\alpha(x,y)|\lesssim k^\alpha_<(x,y)+k^\alpha_>(x,y)
	$$
	with
	\begin{align*}
		\begin{split}
			k^\alpha_>(x,y) :=
			& \int_0^1 ds \int_{z_d> x_d/2} \frac{dz}{z_d^\alpha} \, s^{-\frac{d}{\alpha}}(1-s)^{-\frac{d}{\alpha}} \left(1\wedge\frac{z_d}{(1-s)^{1/\alpha}}\right)^{(\alpha-1)_+} \left(1\wedge\frac{z_d}{s^{1/\alpha}}\right)^{p} \\
			& \quad \times 
			\left[\left(1\wedge\frac{(1-s)^{1+\frac{d}{\alpha}}}{|x-z|^{d+\alpha}}\right)
			\left(1\wedge\frac{s^{1+\frac{d}{\alpha}}}{|y-z|^{d+\alpha}}\right)\one_{\alpha<2}\right.\\
			& \qquad\qquad\qquad\qquad\qquad\qquad\quad \left. + \exp\left(-c\left(\frac{|x-z|^2}{1-s}+\frac{|y-z|^2}{s}\right)\right)\one_{\alpha=2}\right]\,.
		\end{split}
	\end{align*}
	and similarly for $k^\alpha_<$.
	
	\medskip
	
	We discuss $k^\alpha_<$ and $k^\alpha_>$ separately and begin with the latter. We bound $z_d^{-\alpha} \lesssim x_d^{-\alpha}\lesssim y_d^{-\alpha}$ and we bound
	$$
	\left(1\wedge\frac{z_d}{(1-s)^{\frac1\alpha}}\right)^{(\alpha-1)_+} \left(1\wedge\frac{z_d}{s^{1/\alpha}}\right)^{p}
	\leq \left(1\wedge\frac{z_d}{(1-s)^{\frac1\alpha}}\right)^{q} \left(1\wedge\frac{z_d}{s^{1/\alpha}}\right)^{q} \,.
	$$
	Now we enlarge the $z_d$-integration to all of $(0,\infty)$ and reinsert the trivial factors \eqref{eq:trivialfactors}, but with both exponents replaced by $q$. Noting that $q$ is the exponent corresponding to the operator $L_{-\lambda_-}$ (where $\lambda_-=(-\lambda)\vee 0$), we conclude that
	\begin{align*}
		k^\alpha_>(x,y) & \lesssim
		\frac{1}{y_d^\alpha} \int_0^1 ds \int_{\R_+^d}dz\, \me{-(1-s)L_{-\lambda_-}}(x,z) \me{-sL_{-\lambda_-}}(z,y)
		= \frac{1}{y_d^\alpha} \int_0^1 ds\, \me{-L_{-\lambda_-}}(x,y) \\
		& \sim M^\alpha(x,y)\,,
	\end{align*}
	where we used the semigroup property of $\exp(-sL_{-\lambda_-})$ and the heat kernel bounds and we dropped again trivial factors as in \eqref{eq:trivialfactors} (with exponents $q$).
	
	\medskip
	
	It remains to deal with $k^\alpha_<$, where we integrate over $z_d<x_d/2$. We first discuss the case $\alpha<2$.
        We begin by carrying out the $z'$-integration. Computations are simplified if we use the fact that $|x_d-z_d| \sim x_d$ by the choice of the cut-off in the integral and similarly $|y_d-z_d|\sim y_d \sim x_d$ (since $|z_d|\leq x_d/2\leq y_d/2$). Thus, $|x-z|\sim |x'-z'|+x_d$ and $|y-z|\sim |y'-z'|+x_d$ and the integral to be computed is comparable to
	\begin{equation}
		\label{eq:intalphasmaller2}
		\int_{\R^{d-1}}dz' \left(1\wedge\frac{(1-s)^{1+\frac{d}{\alpha}}}{x_d^{d+\alpha} + |x'-z'|^{d+\alpha}}\right)
		\left(1\wedge\frac{s^{1+\frac{d}{\alpha}}}{x_d^{d+\alpha} + |y'-z'|^{d+\alpha}}\right).
	\end{equation}
	We simplify the integrand, using $s\in[0,1]$ and $x_d\sim y_d\geq 1$,
	\begin{equation*}
		1\wedge\frac{(1-s)^{1+\frac{d}{\alpha}}}{x_d^{d+\alpha} + |x'-z'|^{d+\alpha}}
		\sim \frac{(1-s)^{1+\frac{d}{\alpha}}}{x_d^{d+\alpha} + |x'-z'|^{d+\alpha}}
	\end{equation*}
	and
	$$
	1\wedge\frac{s^{1+\frac{d}{\alpha}}}{x_d^{d+\alpha} + |y'-z'|^{d+\alpha}}
	\sim \frac{s^{1+\frac{d}{\alpha}}}{x_d^{d+\alpha} + |y'-z'|^{d+\alpha}} \,.
	$$
	Using Lemma~\ref{integral} we see that the integral \eqref{eq:intalphasmaller2} is comparable to
	$$
	s^{1+\frac d\alpha}(1-s)^{1+\frac d\alpha} \ \frac{x_d^{-\alpha-1}}{x_d^{d+\alpha}+|x'-y'|^{d+\alpha}} \,.
	$$
	For an upper bound, one can remove the term $|x'-y'|$ in the denominator. Thus, we have shown that
	\begin{align*}
          k^\alpha_<(x,y) \lesssim 
          x_d^{-d-2\alpha-1}
          \int_0^1 ds \int_0^{x_d/2} \frac{dz_d}{z_d^\alpha}\, s (1-s) \left(1\wedge\frac{z_d}{(1-s)^{1/\alpha}}\right)^{(\alpha-1)_+} \left(1\wedge\frac{z_d}{s^{1/\alpha}}\right)^{p} \,.
	\end{align*}
		
	Next, we carry out the $s$-integration for fixed $z_d\in[0,x_d/2]$. The integral coming from $s\leq 1/2$ is
	$$
	\sim \int_0^{1/2}ds\, s \left(1\wedge z_d \right)^{(\alpha-1)_+} \left(1\wedge\frac{z_d}{s^{1/\alpha}}\right)^{p} \sim (1 \wedge z_d)^{(\alpha-1)_++p} \,.
	$$
	Here we used $p<\alpha$. Similarly, the integral coming from $s\geq 1/2$ is
	$$
	\sim \int_{1/2}^1 ds \, (1-s) \left(1\wedge\frac{z_d}{(1-s)^{1/\alpha}}\right)^{(\alpha-1)_+} \left(1\wedge z_d \right)^{p} \sim (1 \wedge z_d)^{(\alpha-1)_++p} \,.
	$$
	This leads to the bound
	\begin{align*}
          \begin{split}
            k^\alpha_<(x,y) & \lesssim 
            x_d^{-d-2\alpha-1}
            \int_0^{x_d/2} \frac{dz_d}{z_d^\alpha}\, (1 \wedge z_d)^{(\alpha-1)_++p} \\
            & \sim x_d^{-d-2\alpha-1} \left( \one_{\alpha\geq 1} + (\ln(1+x_d)) \one_{\alpha=1} + x_d^{1-\alpha}\one_{\alpha<1} \right).
          \end{split}
    \end{align*}
	In the last computation, we used the fact that, if $\alpha\geq1$, then $p>0$ (note that for $\alpha=1$, this inequality is ensured by the assumption $\lambda>\lambda_*=0$), and if $\alpha<1$, then $p-\alpha\geq \frac{\alpha-1}{2}-\alpha >-1$.

	Finally, we note that, since $x_d\gtrsim1$,
	$$
	x_d^{-d-2\alpha-1} \left( \one_{\alpha\leq 1} + (\ln(1+x_d)) \one_{\alpha=1} + x_d^{1-\alpha}\one_{\alpha<1} \right)
	\lesssim x_d^{-\alpha} \left( 1 \wedge x_d^{-d-\alpha} \right).
	$$
	Since $x_d\sim y_d$ and $x_d/2\geq |x-y|$, we deduce that $k^\alpha_<(x,y) \lesssim M^\alpha(x,y)$ if $\alpha<2$.
	
	\medskip
	
	It remains to treat the case $\alpha=2$. The argument is similar, but slightly simpler. The $z'$-integral can be done explicitly, yielding
	\begin{align*}
		& \int_{\R^{d-1}}dz'\, \exp\left(-c\left(\frac{|x-z|^2}{1-s}+\frac{|y-z|^2}{s}\right)\right) \\
		& = \const s^{\frac{d-1}{2}} (1-s)^{\frac{d-1}{2}} \exp\left(-c\left(|x'-y'|^2 + \frac{(x_d-z_d)^2}{1-s}+\frac{(y_d-z_d)^2}{s}\right)\right).
	\end{align*}
	For $s\in[0,1]$ and all $x_d,y_d,z_d\geq0$ we bound
	$$
	\frac{(x_d-z_d)^2}{1-s}+\frac{(y_d-z_d)^2}{s} \geq (x_d-z_d)^2 + (y_d-z_d)^2 \geq \frac12 (x_d-y_d)^2 \,.
	$$
	Also, as before, using the restriction $z_d<x_d/2$ and $y_d\geq x_d$,
	$$
	\frac{(x_d-z_d)^2}{1-s}+\frac{(y_d-z_d)^2}{s} \geq (x_d-z_d)^2 + (y_d-z_d)^2 \gtrsim x_d^2 \,.
	$$
	Combining these two bounds gives
	\begin{align*}
		& \exp\left(-c\left(|x'-y'|^2 + \frac{(x_d-z_d)^2}{1-s}+\frac{(y_d-z_d)^2}{s}\right)\right)
		\leq \exp(-\tilde c x_d^2) \exp\left( -\frac{c}4|x-y|^2 \right) \\
		& \lesssim x_d^{-2} \exp\left( -\frac{c}4|x-y|^2 \right) = M^2(x,y) \,,
	\end{align*}
	where $M^2(x,y)$ is now defined with $c$ being one quarter of the constant in the heat kernel bound. (Obviously, the bound on $k_>^2$ remains valid if $c$ is decreased.)
	
	Thus, to prove that $k_<^2(x,y)\lesssim M^2(x,y)$ it suffices to prove that
	\begin{align*}
		\int_0^1 ds \int_0^{x_d/2} \frac{dz_d}{z_d^2}\, s^{-\frac12} (1-s)^{-\frac12} \left(1\wedge\frac{z_d}{(1-s)^{1/2}}\right) \left(1\wedge\frac{z_d}{s^{1/2}}\right)^{p}
		\lesssim 1 \,.
	\end{align*}
	To prove this, we first perform the $s$-integral for fixed $z_d\in[0,x_d/2]$ and find
	\begin{align*}
		& \int_0^1 ds \, s^{-\frac12} (1-s)^{-\frac12} \left(1\wedge\frac{z_d}{(1-s)^{1/2}}\right) \left(1\wedge\frac{z_d}{s^{1/2}}\right)^{p} \\
		& \sim (1\wedge z_d)^{p+1} \left( 1 + \ln(1+\tfrac1{z_d}) \right) \one_{p\leq 1} + (1\wedge z_d)^2 \one_{p>1} \,.
	\end{align*}
	We omit the detail of this computation. Since the right side, multiplied by $z_d^{-2}$, is integrable over $(0,\infty)$ (for $p\leq 1$, we use $p>1/2>0$), we obtain the claimed bound.		
\end{proof}

\section{Proof of the reversed Hardy inequality (Theorem \ref{reversehardy})}
\label{s:schur}

\begin{proof}[Proof of Theorem \ref{reversehardy}]
  \emph{Step 1.}
  The assertion for $s=2$ follows from $L_\lambda-L_0=\lambda x_d^{-\alpha}$.
  In the following we assume $0<s<2$ and $\lambda\neq0$. By the spectral theorem, we have, for $f\in C^\infty_c(\R^d_+)$,
  \begin{align*}
    \left(L_{\lambda}^{s/2} - L_0^{s/2} \right) f 
    & = - \frac{1}{\Gamma(-s/2)} \int_0^\infty \frac{dt}{t}\ t^{-s/2} \left( e^{-tL_0} - e^{-tL_{\lambda}}\right) f \\
    & = - \frac{1}{\Gamma(-s/2)} \int_0^\infty \frac{dt}{t}\ t^{-s/2} \int_{\R_+^d} dy\, K_t^\alpha(\cdot,y) f(y) \,.
  \end{align*}
  (Here we use Lemma \ref{domain}, which guarantees that $C_c^\infty(\R^d_+)\subset \dom L_\lambda^{s/2}\cap \dom L_0^{s/2}$.) Abbreviating $g(y) := y_d^{-\alpha s/2} |f(y)|$, it suffices to show that the right side of
  \begin{align*}
    \begin{split}
      \left\| \left(L_{\lambda}^{s/2} - L_0^{s/2} \right) f \right\|_{L^2(\R_+^d)}
      \lesssim \left\|\int_{\R_+^d}dy\, \int_0^\infty \frac{dt}{t}\ t^{-\frac{s}{2}} K_t^\alpha(\cdot,y) y_d^{\alpha\frac{s}{2}} g(y)\right\|_{L^2(\R_+^d)}
    \end{split}
  \end{align*}
  is bounded by a multiple of $\|g\|_{L^2(\R_+^d)}$. By the pointwise bound of Theorem \ref{differenceheatkernel3} it suffices to show the $L^2(\R_+^d)$-boundedness of the operator associated to the kernel
  \begin{align}
    \label{eq:reversehardyaux1}
    \int_0^\infty dt\, t^{-1-s/2} \left(J_t^\alpha(x,y) + M_t^\alpha(x,y) \right)y_d^{\alpha s/2}\,, \quad x,y\in\R_+^d \,,
  \end{align}
  with $M_t^\alpha$ and $J_t^{\alpha}$ defined in the previous section. This $L^2(\R_+^d)$-boundedness will be shown in the following two steps, which therefore will conclude the proof of Theorem~\ref{reversehardy}.
  
  \medskip
  
  \emph{Step 2.} We begin with the kernel coming from the $M^\alpha_t$-part of \eqref{eq:reversehardyaux1}. As discussed in the proof of Theorem \ref{differenceheatkernel3}, on the support of $M_t^\alpha(x,y)$ we have $x_d\sim y_d$. 
  Hence,
  \begin{align*}
  	\int_0^\infty \frac{dt}{t}\, t^{-\frac{s}{2}} \, M_t^\alpha(x,y) y_d^{\frac{\alpha s}{2}}
  	\sim \int_0^\infty \frac{dt}{t}\, t^{-\frac{s}{2}} \, M_t^\alpha(x,y) (x_d y_d)^{\frac{\alpha s}{4}}.
  \end{align*}
  This replaces the kernel by a symmetric one and we only have to perform a single Schur test instead of two. We obtain
  \begin{align*}
  	& \sup_{x\in\R_+^d} \int_{\R_+^d} dy \int_0^\infty \frac{dt}{t}\, t^{-\frac{s}{2}}\, M_t^\alpha(x,y)  (x_d y_d)^{\frac{\alpha s}{4}} \\
  	& \lesssim \sup_{x\in\R_+^d} \int\limits_{y_d \sim x_d} dy \int\limits_{t\leq(x_d \vee y_d)^\alpha} \frac{dt}{t}\, t^{-\frac{s}{2}}\, (x_d y_d)^{\frac{\alpha s}{4}} \\
  	& \qquad\qquad \times \frac{t^{1-\frac d\alpha}}{(x_d \vee y_d)^\alpha} \left[\left( 1 \wedge \frac{t^{1+\frac d\alpha}}{|x-y|^{d+\alpha}}\right)\one_{\alpha<2} + \exp\left(-c\frac{|x-y|^2}{t}\right) \one_{\alpha=2}\right] \\
  	& \lesssim \sup_{x\in\R_+^d} x_d^{\frac{\alpha s}{2}-\alpha} \int\limits_{y_d \sim x_d} dy \int\limits_{t\lesssim x_d^\alpha} \frac{dt}{t}\, t^{-\frac{s}{2} + 1 - \frac{d}{\alpha}} \, \left( 1 \wedge \frac{t^{1+\frac d\alpha}}{|x-y|^{d+\alpha}} \right)\,.
  \end{align*}
  We now interchange the order of integration and do the $y$-integral first. We bound
  \begin{align*}
  	\int\limits_{y_d\sim x_d} dy\, \left(1\wedge\frac{t^{1+\frac{d}{\alpha}}}{|x-y|^{d+\alpha}}\right)
  	\leq \int_{\R^d} dy \left(1\wedge\frac{t^{1+\frac{d}{\alpha}}}{|x-y|^{d+\alpha}}\right)
  	\sim t^{\frac d\alpha} \,.
  \end{align*}
  Therefore, the supremum over $x\in\R_+^d$ above is $\lesssim \sup_{x\in\R_+^d} x_d^{\frac{\alpha s}{2}-\alpha} \int_0^{Cx_d^\alpha} dt\, t^{-\frac{s}{2}} <\infty$. Thus, the Schur test implies the $L^2(\R^d_+)$-boundedness of the corresponding operator.

	\medskip
	\emph{Step 3.} We now study the kernel coming from the $J^\alpha_t$-part of \eqref{eq:reversehardyaux1}. Two preliminary steps will simplify our computations. First, if $\alpha=2$ we bound $\exp\left(-c\frac{|x-y|^2}{t}\right)\lesssim 1 \wedge \frac{t^{1+\frac{d}{\alpha}}}{|x-y|^{d+\alpha}}$. Second, we replace each of the two factors $(1\wedge \ldots)^q$ by $(1\wedge\ldots)^{-r}$ with
	$$
	-r := q\wedge 0 \,,
	$$
    where we recall $q=\min\{p,(\alpha-1)_+\}$. Thus, $J_t^\alpha\leq \tilde J_t^\alpha$ with
	\begin{align*}
		\begin{split}
			\tilde J_t^{\alpha} (x,y) & := \left( \one_{x_d \vee y_d\leq t^{1/\alpha}} + \one_{x_d \vee y_d\geq t^{1/\alpha}} \one_{|x-y|\geq (x_d\wedge y_d)/2} \right) 
			\left(1\wedge \frac{x_d}{t^{1/\alpha}}\right)^{-r} \,\left(1\wedge \frac{y_d}{t^{1/\alpha}}\right)^{-r} \\
			& \qquad \times t^{-\frac{d}{\alpha}} \left( 1 \wedge \frac{t^{1+\frac{d}{\alpha}}}{|x-y|^{d+\alpha}} \right),
		\end{split}
	\end{align*}
	and it suffices to prove the assertion with $\tilde J^\alpha_t$ instead of $J^\alpha_t$.
	
	For that purpose we insert the cut-offs $\one_{x_d\vee y_d\leq t^{1/\alpha}}$ and $\one_{x_d\vee y_d\geq t^{1/\alpha}}$ and bound the two terms separately.	We have
	\begin{align*}
          \begin{split}
            & \int_0^\infty dt\, t^{-1- \frac s2} \tilde J_t^{\alpha}(x,y)\, \one_{x_d\vee y_d\leq t^{1/\alpha}} \, y_d^{\frac{\alpha s}2} \\
            & \sim y_d^{\frac{\alpha s}2} (x_dy_d)^{-r} \int_{(x_d\vee y_d)^\alpha}^\infty dt\, t^{-1-\frac s2+\frac{2r}{\alpha} -\frac d\alpha} \left( 1\wedge \frac{t^{1+\frac d\alpha}}{|x-y|^{d+\alpha}} \right) \\
            & \lesssim y_d^{\frac{\alpha s}2} (x_dy_d)^{-r} \left[ (|x-y| \! \vee \! x_d \! \vee \! y_d)^{-\frac{\alpha s}2+ 2r-d} + \one_{x_d \vee y_d \leq |x-y|} |x-y|^{-d-\alpha} (x_d \! \vee \! y_d)^{2r+\alpha-\frac{\alpha s}2} \right].
          \end{split}
	\end{align*}
	The first term here comes from the $t$-integral from $(|x-y|\vee x_d\vee y_d)^\alpha$ to $\infty$. This integral converges since $-\frac s2+\frac{2r}\alpha-\frac d\alpha<0$. (Note that $s>0$ and $2r\leq(1-\alpha)_+< 1$.) The second term comes from an upper bound on the integral between $(x_d\vee y_d)^\alpha$ and $|x-y|^\alpha$, in fact, from an upper bound on the integral between $0$ and $|x-y|^\alpha$. This integral converges since $-\frac s2 + \frac{2r}{\alpha}+1>0$.
	
	The above bound can be simplified since (using $r\geq 0$)
	$$
	\one_{x_d \vee y_d \leq |x-y|} |x-y|^{-d-\alpha} (x_d\vee y_d)^{2r+\alpha-\frac{\alpha s}2}
	\leq (|x-y|\vee x_d\vee y_d)^{-\frac{\alpha s}2+ 2r-d} \,.
	$$
	
	We now turn to the contribution to $\tilde J_t^{\alpha}$ from $\{x_d\vee y_d\geq t^{1/\alpha}\}$. We have
	\begin{align*}
		\begin{split}
			& \int_0^\infty dt\, t^{-1- \frac s2} \tilde J_t^{\alpha}(x,y)\, \one_{x_d \vee y_d\geq t^{1/\alpha}} \, y_d^{\frac{\alpha s}2} \\
			& \sim y_d^{\frac{\alpha s}2} \int_0^{(x_d\vee y_d)^\alpha} dt\, t^{-1-\frac s2 -\frac d\alpha} \left( 1 \wedge \frac{x_d\wedge y_d}{t^{1/\alpha}} \right)^{-r} 
			\left( 1\wedge \frac{t^{1+\frac d\alpha}}{|x-y|^{d+\alpha}} \right) \one_{|x-y|\geq (x_d\wedge y_d)/2} \\
			& \leq y_d^{\frac{\alpha s}2} |x-y|^{-d-\alpha} \int_0^{(x_d\vee y_d)^\alpha} dt\, t^{-\frac s2} \left( 1 \wedge \frac{x_d\wedge y_d}{t^{1/\alpha}} \right)^{-r} \one_{|x-y|\geq (x_d\wedge y_d)/2} \\
			& \lesssim y_d^{\frac{\alpha s}2} |x-y|^{-d-\alpha} \left[ (x_d\wedge y_d)^{\alpha-\frac{\alpha s}2} + (x_d\wedge y_d)^{-r} (x_d\vee y_d)^{r+\alpha-\frac{\alpha s}2} \right] \one_{|x-y|\geq (x_d\wedge y_d)/2} \,.
		\end{split}
	\end{align*}
	The first term here comes from the integral from $0$ to $(x_d\wedge y_d)^\alpha$. This converges since $s<2$. The second term comes from an upper bound on the integral from $(x_d\wedge y_d)^\alpha$ to $(x_d\vee y_d)^\alpha$, in fact, from an upper bound on the integral between $0$ and $(x_d\vee y_d)^\alpha$. This integral converges since $-\frac s2+\frac r\alpha>-1$.
	
	The above bound can be simplified since (using $r\geq 0$)
	$$
	(x_d\wedge y_d)^{\alpha-\frac{\alpha s}2} \leq (x_d\wedge y_d)^{-r} (x_d\vee y_d)^{r+\alpha-\frac{\alpha s}2} \,.
	$$
	
	To summarize, we have shown that
	\begin{align*}
          \begin{split}
            \int_0^\infty dt\, t^{-1-\frac s2} \tilde J_t^{\alpha}(x,y) \, y_d^{\frac{\alpha s}2}
            & \lesssim y_d^{\frac{\alpha s}2} (x_dy_d)^{-r} (|x-y|\vee x_d\vee y_d)^{-\frac{\alpha s}2+ 2r-d} \\
            & \quad + y_d^{\frac{\alpha s}2} |x-y|^{-d-\alpha} (x_d\wedge y_d)^{-r} (x_d\vee y_d)^{r+\alpha-\frac{\alpha s}2} \one_{|x-y|\geq (x_d\wedge y_d)/2}.
		\end{split}
	\end{align*}
	We claim that this is
	\begin{equation}
		\label{eq:jintbound}
		\lesssim \left( \frac{|x-y|\vee x_d\vee y_d}{\sqrt{x_d y_d}} \right)^{2r} \frac{(|x-y|\vee x_d\vee y_d)^{\alpha}}{(|x-y|\vee (x_d\wedge y_d))^{d+\alpha}} \,.
	\end{equation}
	Indeed, for the terms involving $s$ this follows from $y_d\leq x_d\vee y_d\leq |x-y|\vee x_d\vee y_d$ and for those involving $r$ it follows from $r\geq 0$ and
	$$
	\frac{|x-y|\vee x_d\vee y_d}{\sqrt{x_d y_d}} \geq \frac{x_d\vee y_d}{\sqrt{x_d y_d}} = \sqrt{ \frac{x_d\vee y_d}{x_d\wedge y_d} } \,.
	$$
	Moreover,
	$$
	(|x-y|\vee x_d\vee y_d)^{-d} \leq \frac{(|x-y|\vee x_d\vee y_d)^{\alpha}}{(|x-y|\vee (x_d\wedge y_d))^{d+\alpha}} \,,
	$$
	and
	$$
	|x-y|^{-d-\alpha} (x_d\vee y_d)^{\alpha-\frac{\alpha s}2} 
	\one_{|x-y|\geq (x_d\wedge y_d)/2}
	\lesssim \frac{(|x-y|\vee x_d\vee y_d)^{\alpha-\frac{\alpha s}{2}}}{(|x-y|\vee (x_d\wedge y_d))^{d+\alpha}} \,.
	$$
	This proves that \eqref{eq:jintbound} is an upper bound on the quantity of interest. The claimed $L^2(\R^d_+)$-boundedness now follows from Proposition \ref{schurkillip2} below, noting that $r\leq(1-\alpha)_+/2<1/2$.
\end{proof}

\begin{proposition}
  \label{schurkillip2}
  Let $\alpha>0$ and $0\leq r<\frac12$. Then the integral operator with integral kernel
  \begin{align*}
    \left( \frac{|x-y|\vee x_d\vee y_d}{\sqrt{x_d y_d}} \right)^{2r} \frac{(|x-y|\vee x_d\vee y_d)^{\alpha}}{(|x-y|\vee (x_d\wedge y_d))^{d+\alpha}}
  \end{align*}
  is bounded on $L^2(\R_+^d)$.
\end{proposition}

\begin{proof}
	\emph{Step 1.}
	We denote the kernel in the proposition by $k(x,y)$. As a preliminary step to the main argument, let us carry out the integration over the $\R^{d-1}$-variables. We claim that
	\begin{equation}
		\label{eq:schurmarginal}
		\int_{\R^{d-1}} dy'\, k(x,y) \lesssim \left( \frac{x_d\vee y_d}{\sqrt{x_d y_d}} \right)^{2r} \frac{(x_d\vee y_d)^{\alpha}}{(|x_d-y_d|\vee (x_d\wedge y_d))^{1+\alpha}} \,.
	\end{equation}
	Note that the kernel on the right side is the kernel corresponding to the case $d=1$ of the proposition. (Indeed, one has $|x_d-y_d|\leq x_d\vee y_d$, so $x_d\vee y_d = |x_d-y_d|\vee x_d\vee y_d$.)
	
	To prove \eqref{eq:schurmarginal}, we distinguish between the regions where $|x-y|\gtrless x_d\vee y_d$. We find
	\begin{align}
		\label{eq:schurmarginal2}
		\begin{split}
			\int_{\R^{d-1}} dy'\, k(x,y) & \lesssim \left( \frac{x_d\vee y_d}{\sqrt{x_d y_d}} \right)^{2r} \int_{|x'-y'|<x_d\vee y_d}dy'\, \frac{(x_d\vee y_d)^\alpha}{(|x'-y'|\vee|x_d-y_d|\vee(x_d\wedge y_d))^{d+\alpha}} \\
			& \quad + \int_{|x-y|>x_d\vee y_d} dy'\, \left( \frac{|x-y|}{\sqrt{x_d y_d}} \right)^{2r} \frac{1}{|x-y|^d} \,.
		\end{split}
	\end{align}
	In the first integral we scale $y'=x'+(|x_d-y_d|\vee(x_d\wedge y_d))w$ and obtain
	\begin{align*}
		& \left( \frac{x_d\vee y_d}{\sqrt{x_d y_d}} \right)^{2r} \int_{|x'-y'|<x_d\vee y_d}dy'\, \frac{(x_d\vee y_d)^\alpha}{(|x'-y'|\vee|x_d-y_d|\vee(x_d\wedge y_d))^{d+\alpha}} \\
		& = \left( \frac{x_d\vee y_d}{\sqrt{x_d y_d}} \right)^{2r} \frac{(x_d\vee y_d)^\alpha}{(|x_d-y_d|\vee(x_d\wedge y_d))^{1+\alpha}}
		\int_{|w|<(x_d\vee y_d)/(|x_d-y_d|\vee(x_d\wedge y_d))} \frac{dw}{(|w|\vee 1)^{d+\alpha}} \,.
	\end{align*}
	Bounding the latter integral by a constant, we obtain a term of the form \eqref{eq:schurmarginal}.
	
	We turn now to the second integral in \eqref{eq:schurmarginal2} and claim that
	\begin{equation}
		\label{eq:schurmarginal3}
		\int_{|x-y|>x_d\vee y_d} dy'\, \left( \frac{|x-y|}{\sqrt{x_d y_d}} \right)^{2r} \frac{1}{|x-y|^d} \lesssim \left( \frac{x_d\vee y_d}{\sqrt{x_d y_d}} \right)^{2r} \frac{1}{x_d\vee y_d} \,.
	\end{equation}
	Since $|x_d-y_d|\vee (x_d\wedge y_d)\leq x_d\vee y_d$, this will prove \eqref{eq:schurmarginal}.
	 
	To prove \eqref{eq:schurmarginal3} we first restrict the integral to $|x'-y'|>(x_d\vee y_d)/2$ and find, changing variables $y'=x'+|x_d-y_d|w$,
	\begin{align*}
		& \int_{|x'-y'|>(x_d\vee y_d)/2} dy'\, \left( \frac{|x-y|}{\sqrt{x_d y_d}} \right)^{2r} \frac{1}{|x-y|^d} \\ 
		& = \frac{|x_d-y_d|^{2r-1}}{(x_dy_d)^r} \int_{|w|>(x_d\vee y_d)/(2|x_d-y_d|)} \frac{dw}{(1+|w|^2)^{(d-2r)/2}} \\
		& \sim \left( \frac{x_d\vee y_d}{\sqrt{x_dy_d}} \right)^{2r} (x_d\vee y_d)^{-1} \,.
	\end{align*}
	Here we used $r<\frac12$. This bound is of the form \eqref{eq:schurmarginal3}.
	
	It remains to compute the integral in \eqref{eq:schurmarginal3} where the restriction $|x-y|>x_d\vee y_d$ is replaced by $|x'-y'|\leq (x_d\vee y_d)/2$. In the latter region we have
	$$
	\tfrac14(x_d\vee y_d)^2  + |x_d-y_d|^2 \geq |x-y|^2 \geq (x_d\vee y_d)^2 \,,
	$$
	and therefore $|x_d-y_d|\gtrsim x_d\vee y_d$. Clearly $|x_d-y_d|\leq x_d\vee y_d$ and therefore $|x-y|\sim x_d\vee y_d$. Thus,
	\begin{align*}
		& \int_{2|x'-y'|\leq x_d\vee y_d <|x-y|} dy'\, \left( \frac{|x-y|}{\sqrt{x_d y_d}} \right)^{2r} \frac{1}{|x-y|^d} \\
		& \sim \left( \frac{x_d\vee y_d}{\sqrt{x_d y_d}} \right)^{2r} \frac{1}{(x_d\vee y_d)^d} \int_{2|x'-y'|\leq x_d\vee y_d <|x-y|} dy' \\
		& \lesssim \left( \frac{x_d\vee y_d}{\sqrt{x_d y_d}} \right)^{2r} \frac{1}{x_d\vee y_d} \,,
	\end{align*}
	which is again of the form \eqref{eq:schurmarginal3}. This completes the proof of \eqref{eq:schurmarginal}.
	
	\medskip
	
	\emph{Step 2.}	
	We perform weighted Schur tests for the operator with kernel given by the right side of \eqref{eq:schurmarginal}. As weight we choose
	$$
	w(x,y)=\left( \frac{x_d}{y_d} \right)^\beta
	\qquad\text{with}\ r<\beta<1-r \,.
	$$
	Since $r<\frac12$, it is possible to find such a $\beta$.
	
	For the first part of the Schur test, we use \eqref{eq:schurmarginal} to bound
	\begin{align*}
		\int_{\R^d_+} dy\, w(x,y) k(x,y) & \sim \int_0^\infty dy_d\, \left( \frac{x_d}{y_d} \right)^\beta \left( \frac{x_d\vee y_d}{\sqrt{x_d y_d}} \right)^{2r} \frac{(x_d\vee y_d)^{\alpha}}{(|x_d-y_d|\vee (x_d\wedge y_d))^{1+\alpha}} \\
		& = \int_0^\infty dt\, t^{-\beta -r} \frac{(1\vee t)^{\alpha+2r}}{(|1-t|\vee (1\wedge t))^{1+\alpha}} \\
		& \sim \int_0^\infty dt\, t^{-\beta-r} (1\wedge t^{-1+2r}) <\infty \,.
	\end{align*}
	The finiteness of the last integral uses the assumptions $r<\beta<1-r$.

	For the second part of the Schur test, we note that, by symmetry, \eqref{eq:schurmarginal} remains valid with $dy'$ replaced by $dx'$. Thus,
	\begin{align*}
		\int_{\R^d_+} dx\, w(x,y)^{-1} k(x,y) & \sim \int_0^\infty dx_d\, \left( \frac{y_d}{x_d} \right)^\beta \left( \frac{x_d\vee y_d}{\sqrt{x_d y_d}} \right)^{2r} \frac{(x_d\vee y_d)^{\alpha}}{(|x_d-y_d|\vee (x_d\wedge y_d))^{1+\alpha}} \\
		& = \int_0^\infty dt\, t^{-\beta -r} \frac{(1\vee t)^{\alpha+2r}}{(|1-t|\vee (1\wedge t))^{1+\alpha}} <\infty \,,
	\end{align*}
	as before. The $L^2(\R^d_+)$-boundedness therefore follows from the Schur test.
\end{proof}

\section{Commutator bounds}\label{sec:comm}

Throughout this section we assume that $0<\alpha<2$. Our goal is to bound the commutators
$$
[(-\Delta)^{\alpha/2},\zeta] v(x) = \mathcal A(d,-\alpha) \int_{\R^d} \frac{\zeta(x)-\zeta(y)}{|x-y|^{d+\alpha}} v(y)\,dy
$$
for functions $v$ supported in $\overline{\R^d_+}$. In general the integral on the right side does not converge absolutely and should be understood as a principal value integral (whose converges we will follow from our results).

We will impose certain boundedness and decay assumptions on $v$, as well as, for $\alpha\geq 1$, mild regularity assumptions. The function $\zeta$ is a cut-off function and we are interested in tracking the dependence of the commutator on the size of the transition zone, where $\zeta$ switches from zero to one.

This section is split into three parts, corresponding to different choices of the cut-off function $\zeta$. In Subsection \ref{sec:cutoffrad} we will consider a cut off at a large distance from the origin, in Subsection \ref{sec:cutoffbdry} a cut off at a small distance from the boundary hyperplane, and in Subsection \ref{sec:cutoffcomb} the combination of both.

The assumption on $v$ will always be of the form
\begin{equation}
	\label{eq:assv1}
	|v(x)|\leq (1\wedge |x|^{-d-\alpha}) (1\wedge x_d)^p 
	\qquad\text{for all}\ x\in\R^d_+ \, 
\end{equation}
with a certain parameter $p\geq\frac{\alpha-1}{2}$. This bound is reminiscent of the heat kernel bound in Theorem \ref{heatkernel} and, in fact, in the next section we will use this theorem to verify \eqref{eq:assv1} in our application where $v\in\me{-tL_\lambda}C_c^\infty(\R_+^d)$. There, the parameter $p$ will depend on $\lambda$ as in our main result, but in this section $p$ is an arbitrary parameter.

The additional regularity assumptions will be formulated in terms of the following H\"older seminorms. For a function $u$ on a set $\Omega$ and $0<\beta\leq 2$, we write
\begin{equation}
  \label{eq:holder}
  [u]_{C^\beta(\Omega)} :=
  \begin{cases}
    \sup_{x,y\in\Omega} \frac{|u(x)-u(y)|}{|x-y|^\beta} & \text{if}\ 0<\beta\leq 1 \,, \\
    \sup_{x,y\in\Omega} \frac{|\nabla u(x)-\nabla u(y)|}{|x-y|^{\beta-1}} & \text{if}\ 1<\beta\leq 2 \,.
  \end{cases}
\end{equation}
Our assumption on $v$ will then read
\begin{equation}
	\label{eq:assv2}
	[v]_{C^\beta(B_{\ell_x}(x))} \leq (1\wedge |x|^{-d-\alpha}) \, (1\wedge x_d)^{p-\beta}
	\qquad\text{for all}\ x\in\R^d_+\ \text{with}\ \ell_x := 1\wedge \tfrac{x_d}2 \,.
\end{equation}
We will always assume that $\beta>\alpha-1$.

\subsection{Radial cut-off}\label{sec:cutoffrad}

In this subsection we bound the term
$$
I(x) := \int_{\R^d} \frac{\chi(x)-\chi(y)}{|x-y|^{d+\alpha}}v(y)\,dy \,,
$$
where $v$ is supported in $\overline{\R^d_+}$. Concerning the function $\chi$ we assume that, for a certain parameter $R\geq 1$,
\begin{equation}
	\label{eq:asschi1}
	0\leq\chi\leq 1 \,,
	\qquad
	\chi(x) = 1 \ \text{if}\ |x|\leq R \,,
	\qquad
	\chi(x) = 0 \ \text{if}\ |x|\geq 2R \,,
	\qquad
	|\nabla\chi|\lesssim R^{-1} \,,
\end{equation}
as well as, if $\alpha\geq 1$,
\begin{equation}
  \label{eq:asschi2}
  |D^2\chi|\lesssim R^{-2} \,.
\end{equation}
Here $D^2\chi$ denotes the Hessian of $\chi$.

\begin{lemma}\label{radial1}
  Let $0<\alpha<2$. Let $R\geq 1$, assume that $\chi$ satisfies \eqref{eq:asschi1} and, if $\alpha\geq 1$ also \eqref{eq:asschi2}. Let $p\geq \tfrac{\alpha-1}2$, assume that $v$ satisfies \eqref{eq:assv1} and, if $\alpha\geq 1$, also \eqref{eq:assv2} with some $\beta>\alpha-1$.
	\begin{itemize}
		\item[(a)] If $\alpha<1$, then
		$$
		|I(x)| \lesssim \one_{|x|\leq R} R^{-d-2\alpha} + \one_{|x|>R} |x|^{-d-\alpha}
		\qquad\text{for all}\ x\in\R^d_+ \,.
		$$
		\item[(b)] If $\alpha\geq 1$, then
		\begin{align*}
			|I(x)| & \lesssim \one_{|x|\leq R} R^{-d-2\alpha} + \one_{|x|>R} |x|^{-d-\alpha} \\
			& \quad + \one_{|x|\sim R}
			R^{-d-\alpha-1} \left( (1\wedge x_d)^{-(p-\alpha+1)_-} + \one_{p=\alpha-1}\ln\tfrac1{1\wedge x_d} + \one_{\alpha=1} \ln R \right) \\
			& \quad \text{for all}\ x\in\R^d_+ \,.
		\end{align*}
		\item[(c)] In either case,
		$$
		\| I \|_{L^2(\R^d_+)} \lesssim R^{-\alpha-d/2} \,.
		$$
	\end{itemize}
\end{lemma}

In the formulation of (b) we recall the notation $a_-:=\max\{-a,0\}$.

\begin{proof}[Proof of Lemma~\ref{radial1}. Case $\alpha<1$.]
	\emph{Step 1.} We claim that
	\begin{align}\label{eq:radial1}
		|I(x)| & \lesssim \one_{|x|\leq 4R} \int_{|y|>R} \frac{1}{|y|^{d+\alpha}} |v(y)| \,dy + \one_{|x|>\tfrac R2} \frac{1}{|x|^{d+\alpha}} \int_{|y|\leq 2R} |v(y)| \,dy \notag \\
		& \quad + \one_{\tfrac R2<|x|\leq 4R} \frac 1R \int_{\tfrac R4<|y|\leq 8R} \frac{1}{|x-y|^{d+\alpha-1}} |v(y)| \,dy \,.
	\end{align}
	
	To prove \eqref{eq:radial1} we note that, if $|x|\leq\tfrac R2$, then
	$$
	|I(x)| = \left| \int_{|y|>R} \ldots \right| \lesssim \int_{|y|>R} \frac{1}{|y|^{d+\alpha}} |v(y)| \,dy \,.
	$$
	If $|x|> 4R$, then
	$$
	|I(x)| = \left| \int_{|y|\leq 2R} \ldots \right| \lesssim \int_{|y|\leq 2R} \frac{1}{|x|^{d+\alpha}}|v(y)|\,dy \,.
	$$
	Finally, if $\tfrac R2<|x|\leq 4R$, then
	\begin{align*}
          |I(x)|
          & \lesssim \int_{|y|>8R} \frac{1}{|y|^{d+\alpha}} |v(y)| \,dy + \int_{|y|\leq \tfrac R4} \frac{1}{|x|^{d+\alpha}} |v(y)| \,dy \\
          & \quad + \frac1R \int_{\tfrac R4<|y|\leq 8R} \frac{1}{|x-y|^{d+\alpha-1}} |v(y)| \,dy \,,
	\end{align*}
	where we used $| \chi(x) - \chi(y)|\lesssim R^{-1}|x-y|$, which follows from the gradient bound on $\chi$. Combining the above bounds, we obtain \eqref{eq:radial1}.
	
	\medskip
	
	\emph{Step 2.} We now insert the bounds on $v$ into the right side of \eqref{eq:radial1}. We clearly have
	$$
	\int_{|y|\leq 2R} (1\wedge |y|^{-d-\alpha})(1\wedge y_d)^p \,dy \lesssim 1
	$$
	and
	$$
	\int_{|y|>R} \frac1{|y|^{d+\alpha}} (1\wedge |y|^{-d-\alpha})(1\wedge y_d)^p \,dy \lesssim R^{-d-2\alpha} \,.
	$$
	The last bound is clear if $p\geq 0$ (which is the only case relevant when $\lambda\geq 0$). When $p<0$ the same bound is valid for the integral restricted to $y_d\geq 1$. For the integral with the opposite restriction is easily seen to be bounded by $R^{-d-1-2\alpha}$. (Note that in this integral one has $|y'|\sim |y|$.) Finally, if $\tfrac R2<|x|\leq 4R$, then, since $\alpha\in(0,1)$,
	\begin{align*}
          & \int_{\tfrac R4<|y|\leq 8R} \frac{1}{|x-y|^{d+\alpha-1}} (1\wedge |y|^{-d-\alpha})(1\wedge y_d)^p \,dy \\
          & \quad \lesssim R^{-d-\alpha} \int_{\tfrac R4<|y|\leq 8R} \frac{(1\wedge y_d)^p}{|x-y|^{d+\alpha-1}} \,dy
            \lesssim R^{-d-2\alpha+1} \,.
	\end{align*}
	Here, for an upper bound, we replace the integral over $\{\tfrac R4<|y|\leq 8R\}$ by the integral over $|x-y|\leq 12 R$. For $p\geq 0$ we can drop the factor $(1\wedge y_d)^p$. For $p<0$ we argue similarly as before by distinguishing the cases $y_d\leq 1$ and $y_d>1$.
	
	This proves the claimed pointwise bound in (a). The $L^2$-bound in (c) follows by a simple integration.
\end{proof}

It remains to prove Lemma \ref{radial1} for $\alpha\geq 1$. We discuss the first part of the argument in greater generality since it will also be useful in the next subsection. We are interested in bounding
$$
\int_{\R^d} \frac{\zeta(x)-\zeta(y)}{|x-y|^{d+\alpha}} v(y)\,dy \,,
$$
where $\zeta$ is $C^2$ and $v$ is H\"older continuous with some exponent $\beta$. In the setting of Lemma \ref{radial1} we have $\zeta=\chi$.

We fix a local length scale $\ell_x$, depending on $x\in\R^d_+$, and we decompose
\begin{align}
	\label{eq:commutator1lalpha}
	\begin{split}
		\int_{\R^d_+} \frac{\zeta(x)-\zeta(y)}{|x-y|^{d+\alpha}} v(y)\,dy
		& = \int_{|y-x|\leq\ell_x} \frac{\zeta(x)-\zeta(y)}{|x-y|^{d+\alpha}} (v(y)-v(x))\,dy \\
		& \quad + v(x) \int_{|y-x|\leq\ell_x} \frac{\zeta(x)-\zeta(y)+\nabla\zeta(x)\cdot(y-x)}{|x-y|^{d+\alpha}} \,dy \\
		& \quad + \int_{|y-x|>\ell_x} \frac{\zeta(x)-\zeta(y)}{|x-y|^{d+\alpha}} v(y)\,dy \,.
	\end{split}
\end{align}
Note that because of the principal value we were free to introduce the term $\nabla\zeta(x)\cdot(y-x)$, which contributes zero to the integral (because of oddness), but makes it converge absolutely. We will always bound the first term by
\begin{align}\label{eq:commutator1lalpha1}
  \begin{split}
    & \left| \int_{|y-x|\leq\ell_x} \frac{\zeta(x)-\zeta(y)}{|x-y|^{d+\alpha}} (v(y)-v(x))\,dy \right| \\
    & \quad \leq [v]_{C^\beta(B_{\ell_x}(x))} [\zeta]_{C^1(B_{\ell_x}(x))} \int_{|y-x|\leq\ell_x} \frac{dy}{|x-y|^{d+\alpha-1-\beta}} \,dy \\
    & \quad \lesssim [v]_{C^\beta(B_{\ell_x}(x))} [\zeta]_{C^1(B_{\ell_x}(x))} \ell_x^{-\alpha+1+\beta}
  \end{split}
\end{align}
for some $\beta>\alpha-1$. Similarly, we will bound the second term by
\begin{align}\label{eq:commutator1lalpha2}
  \begin{split}
    & \left| v(x) \int_{|y-x|\leq\ell_x} \frac{\zeta(x)-\zeta(y)+\nabla\zeta(x)\cdot(y-x)}{|x-y|^{d+\alpha}} \,dy \right| \\
    & \quad \leq |v(x)| [\zeta]_{C^2(B_{\ell_x}(x))} \int_{|y-x|\leq\ell_x} \frac{dy}{|x-y|^{d+\alpha-2}} \,dy
      \lesssim |v(x)| [\zeta]_{C^2(B_{\ell_x}(x))} \ell_x^{2-\alpha} \,.
  \end{split}
\end{align}

After these preliminaries we return to the proof of Lemma \ref{radial1}.

\begin{proof}[Proof of Lemma~\ref{radial1}. Case $\alpha\geq 1$.]
	We apply the preceding discussion with $\zeta=\chi$. For the first term in \eqref{eq:commutator1lalpha} we use the bound \eqref{eq:commutator1lalpha1} and note that $[\chi]_{C^1(B_{\ell_x}(x))}$ vanishes unless $|x|\sim R$, in which case it is $\lesssim R^{-1}$. This leads to a bound
	$$
	\one_{|x|\sim R} (1\wedge |x|^{-d-\alpha}) (1\wedge x_d)^{p-\alpha+1} R^{-1} \,. 
	$$
	Similarly, for the second term in \eqref{eq:commutator1lalpha} using the bound \eqref{eq:commutator1lalpha2} we obtain
	$$
	\one_{|x|\sim R} (1\wedge |x|^{-d-\alpha}) (1\wedge x_d)^{p-\alpha+2} R^{-2} \,.
	$$
	Since $1\wedge x_d \leq 1\leq R$, this bound on the second term is smaller than the bound on the first term and can be ignored.
	
	We now turn to the third term in \eqref{eq:commutator1lalpha},
	$$
	\widetilde{I}(x) := \int_{|y-x|>\ell_x} \frac{\chi(x)-\chi(y)}{|x-y|^{d+\alpha}} v(y)\,dy \,.
	$$
	We claim that
	\begin{align}\label{eq:radial2}
          \left| \widetilde{I}(x) \right| & \lesssim \one_{|x|\leq 4R} \int_{|y|>R} \frac{1}{|y|^{d+\alpha}} |v(y)| \,dy + \one_{|x|>\tfrac R2} \frac{1}{|x|^{d+\alpha}} \int_{|y|\leq 2R} |v(y)|\,dy \notag \\
                                          & \quad + \one_{\tfrac R2<|x|\leq 4R} \frac1R \int_{\tfrac R4<|y|\leq 8R} \frac{\one_{|x-y|>\ell_x}}{|x-y|^{d+\alpha-1}}|v(y)|\,dy \,.
	\end{align}
	This is proved in the exact same way as \eqref{eq:radial1}.
	
	We now insert the bounds on $v$ into the right side of \eqref{eq:radial2}. The first two terms are bounded as in the case $\alpha<1$. The bound for the third term in \eqref{eq:radial2}, however, is different now, since $|x-y|^{-d-\alpha+1}$ is not locally integrable. We claim that
	\begin{align*}
          & \int_{\tfrac R4<|y|\leq 8R} \frac{\one_{|x-y|>\ell_x}}{|x-y|^{d+\alpha-1}}|v(y)|\,dy \\
          & \quad \lesssim R^{-d-\alpha} \left( (1\wedge x_d)^{-(p-\alpha+1)_-} + \one_{p=\alpha-1}\ln\tfrac1{1\wedge x_d} + \one_{\alpha=1} \ln R \right).
	\end{align*}
	Indeed, the factor of $R^{-d-\alpha}$ comes from one factor in the bound on $v$, so it suffices to prove
	\begin{align*}
          & \int_{\tfrac R4<|y|\leq 8R} \frac{\one_{|x-y|>\ell_x}}{|x-y|^{d+\alpha-1}}(1\wedge y_d)^p \,dy \\
          & \quad \lesssim (1\wedge x_d)^{-(p-\alpha+1)_-} + \one_{p=\alpha-1}\ln\tfrac1{1\wedge x_d} + \one_{\alpha=1} \ln R \,.
	\end{align*}
	We split the $y$-integral according to whether $y_d>3$ or $y_d\leq 3$. Beginning with the former case, we note that $|x-y|\geq |x_d-y_d|> 1$ if $x_d\leq 2$. Thus, for all $x_d>0$,
	\begin{align*}
          \int_{\tfrac R4<|y|\leq 8R} \frac{\one_{|x-y|>\ell_x}\one_{y_d>3}}{|x-y|^{d+\alpha-1}}(1\wedge y_d)^p \,dy
          & \lesssim \int_{\tfrac R4<|y|\leq 8R} \frac{\one_{|x-y|>1}}{|x-y|^{d+\alpha-1}} \,dy \\
          & \leq 1 + (\ln R)\one_{\alpha=1} \,. 
	\end{align*}
	Next, we consider the integral over $y_d\leq 3$. Performing the $y'$-integration over all of $\R^{d-1}$, we obtain
	\begin{align*}
          & \int_{\tfrac R4<|y|\leq 8R} \frac{\one_{|x-y|>\ell_x} \one_{y_d \leq 3}}{|x-y|^{d+\alpha-1}}(1\wedge y_d)^p \,dy \\
          & \quad \lesssim \int_{y_d\leq 3} \frac{y_d^p}{|x'-y'|^{d+\alpha-1}+|x_d-y_d|^{d+\alpha-1}+\ell_x^{d+\alpha-1}}\,dy
            \lesssim \int_0^3 \frac{y_d^p}{|x_d-y_d|^{\alpha}+\ell_x^{\alpha}}\,dy_d
	\end{align*}
	This integral is easily seen to be
	$$
	\lesssim x_d^{-\alpha}\one_{x_d>1} + x_d^{-(p-\alpha+1)_-}\one_{x_d\leq 1} + \ln\frac{1}{1\wedge x_d}\one_{p=\alpha-1} \,.
	$$
	(To prove this for $x_d\leq 3/2$ we split the $y_d$-integral at $y_d=2x_d$.) This proves the claimed bound.
	
	Combining all these bounds, we obtain the claimed pointwise bound on $I$ in statement (b) of Lemma~\ref{radial1}. The $L^2$-bound in (c) follows by integration. On easily verifies that the `additional' term (compared to the case $\alpha<1$) is subdominant. Here we note that $2(p-\alpha+1)>-1$, which makes the relevant $x_d$-integral finite near the origin.
\end{proof}


\subsection{Boundary cut-off}\label{sec:cutoffbdry}

In this subsection we bound the term
$$
II(x) := \int_{\R^d} \frac{\theta(x)-\theta(y)}{|x-y|^{d+\alpha}}v(y)\,dy \,.
$$
As before, the function $v$ will be supported in $\overline{\R^d_+}$. Concerning the function $\theta$ we assume that, for a certain parameter $r\leq 1$,
\begin{equation}
	\label{eq:asstheta1}
	0\leq\theta\leq 1 \,,
	\qquad
	\theta(x) = 0 \ \text{if}\ x_d\leq r \,,
	\qquad
	\theta(x) = 1 \ \text{if}\ x_d\geq 2r \,,
	\qquad
	|\nabla\theta|\lesssim r^{-1} \,,
\end{equation}
as well as, if $\alpha\geq 1$ and $d=1$,
\begin{equation}
  \label{eq:asstheta2}
  |D^2\theta|\lesssim r^{-2} \,.
\end{equation}
To simplify matters, we assume that $\theta$ is only a function of the last coordinate $x_d$ of $x=(x',x_d)$.

\begin{lemma}\label{bdry1}
	Let $0<\alpha<2$. Let $r\leq 1$ and assume that $\theta$ satisfies \eqref{eq:asstheta1} and, if $\alpha\geq 1$ and $d=1$, also \eqref{eq:asstheta2}. Let $\frac{\alpha-1}{2}\leq p<\alpha$, assume that $v$ satisfies \eqref{eq:assv1} and, if $\alpha\geq 1$ and $d=1$, also \eqref{eq:assv2} with some $\beta>\alpha-1$. Then
	$$
		|II(x)| \lesssim \left(r^{p-\alpha}\wedge \frac{r^{p+1}}{x_d^{1+\alpha}}\right) (1+x_d)^{1+\alpha} (1\wedge |x|^{-d-\alpha})
		\qquad\text{for all}\ x\in\R^d_+ \,.
	$$
	In particular
	$$
		\| II \|_{L^2(\R^d_+)} \lesssim r^{p-\alpha+1/2} \,.
	$$
\end{lemma}

\begin{proof}[Proof of Lemma~\ref{bdry1}. Case $d\geq 2$ or $d\geq 1$ and $\alpha<1$]
	\emph{Step 1.} We claim that
	\begin{align}
		\label{eq:bdry1}
		|II(x)| & \lesssim \one_{x_d\leq 4r} \int_{y_d>r} \frac{1}{|x'-y'|^{d+\alpha}+y_d^{d+\alpha}} |v(y)| \,dy \\
		& \quad + \one_{x_d>\tfrac r2} \int_{y_d\leq 2r} \frac{1}{|x'-y'|^{d+\alpha}+x_d^{d+\alpha}} |v(y)| \,dy \notag \\
		& \quad + \one_{\tfrac r2<x_d\leq 4r} \frac 1r \int_{\tfrac r4<y_d \leq 8r} \frac{|x_d-y_d|}{|x'-y'|^{d+\alpha}+|x_d-y_d|^{d+\alpha}} |v(y)| \,dy \notag \,.
	\end{align}
	
	To prove \eqref{eq:bdry1} we argue in the same way as we did for \eqref{eq:radial1}. We 
	note that, if $x_d\leq\tfrac r2$, then
	$$
	|II(x)| = \left| \int_{y_d>r} \ldots \right| \lesssim \int_{y_d>r} \frac{1}{|x'-y'|^{d+\alpha}+y_d^{d+\alpha}} |v(y)| \,dy \,.
	$$
	If $x_d> 4r$, then
	$$
	|II(x)| = \left| \int_{y_d\leq 2r} \ldots \right| \lesssim \int_{y_d \leq 2r} \frac{1}{|x'-y'|^{d+\alpha}+x_d^{d+\alpha}}|v(y)| \,dy \,.
	$$
	Finally, if $\tfrac r2<x_d \leq 4r$, then
	\begin{align*}
		|II(x)| & \lesssim \int_{y_d>8r} \frac{1}{|x'-y'|^{d+\alpha}+y_d^{d+\alpha}} |v(y)| \,dy + \int_{y_d \leq \tfrac r4} \frac{1}{|x'-y'|^{d+\alpha}+x_d^{d+\alpha}} |v(y)| \,dy \\
		& \quad + \frac1r \int_{\tfrac r4<y_d \leq 8r} \frac{|x_d-y_d|}{|x'-y'|^{d+\alpha}+|x_d-y_d|^{d+\alpha}} |v(y)| \,dy \,,
	\end{align*}
	where we used $| \theta(x) - \theta(y)|\lesssim r^{-1}|x_d-y_d|$, which follows from the gradient bound on $\theta$ and the fact that it only depends on the last coordinate. Combining the above bounds, we obtain \eqref{eq:bdry1}.
	
	\medskip
	
	\emph{Step 2.} We now insert the bounds on $v$ into the right side of \eqref{eq:bdry1}. In the two integrals with an upper bound on $y_d$ we use $1\wedge |y|^{-d-\alpha}\sim 1\wedge |y'|^{-d-\alpha}$. This allows us to compute the $y_d$-integral in the second integral. In this way, we obtain
	\begin{align}\label{eq:bdry1a}
		|II(x)| & \lesssim \one_{x_d\leq 4r} \int_{y_d>r} \frac{1}{|x'-y'|^{d+\alpha}+y_d^{d+\alpha}}(1\wedge |y|^{-d-\alpha})(1\wedge y_d)^p \,dy \notag \\
		& \quad + \one_{x_d>\tfrac r2} r^{p+1} \int_{\R^{d-1}} \frac{1}{|x'-y'|^{d+\alpha}+x_d^{d+\alpha}} (1\wedge |y'|^{-d-\alpha}) \,dy' \notag \\
		& \quad + \one_{\tfrac r2<x_d \leq 4r} r^{p-1} \int_{\tfrac r4<y_d \leq 8r} \frac{|x_d-y_d|}{|x'-y'|^{d+\alpha}+|x_d-y_d|^{d+\alpha}} (1\wedge |y'|^{-d-\alpha}) \,dy \,.
	\end{align}
	A straightforward computation shows that, if $\frac r2<x_d\leq 4r$, then
	\begin{equation}
		\label{eq:integralxd}
		\int_{\tfrac r4<y_d \leq 8r} \frac{|x_d-y_d|}{|x'-y'|^{d+\alpha}+|x_d-y_d|^{d+\alpha}} \,dy_d \sim \frac{r^2}{|x'-y'|^{d+\alpha}+x_d^{d+\alpha}} \,.
	\end{equation}
	(Indeed, we substitute $y_d=x_d+|x'-y'| t$ and note that the upper and lower bounds in the $t$ integral are of order $r$.) We note that \eqref{eq:integralxd} requires the assumption $\alpha<1$ if $d=1$ (with the convention that terms involving $x'$ or $y'$ are absent).
	
	If we substitute \eqref{eq:integralxd} into \eqref{eq:bdry1a}, we see that the third term on the right side of \eqref{eq:bdry1a} is bounded from above by a constant times the second term and can therefore be dropped.
	
	We now perform the $y'$ integral in the first and second integrals in \eqref{eq:bdry1a} using Lemma \ref{integral} below. (We note that $1\wedge |y'|^{-d-\alpha} \sim (1+|y'|^{d+\alpha})^{-1}$ and $1\wedge |y|^{-d-\alpha} \sim ((1+y_d)^{d+\alpha} + |y'|^{d+\alpha})^{-1}$.) In this way, we obtain
	\begin{align*}
		|II(x)| & \lesssim \one_{x_d\leq 4r} \int_{y_d>r}
		\frac{y_d^{-1-\alpha}}{(1+y_d)^{d+\alpha}+|x'|^{d+\alpha}}(1\wedge y_d)^p \,dy_d + \one_{x_d>\tfrac r2} r^{p+1} \frac{x_d^{-1-\alpha}(1+x_d)^{1+\alpha}}{(1+x_d)^{d+\alpha} + |x'|^{d+\alpha}} \,.
	\end{align*}
	Finally, we compute
	$$
	\int_{y_d>r}
	\frac{y_d^{-1-\alpha}}{(1+y_d)^{d+\alpha}+|x'|^{d+\alpha}}(1\wedge y_d)^p \,dy_d
	\sim \frac{r^{p-\alpha}}{1+|x'|^{d+\alpha}} \,.
	$$
	The dominant contribution comes from the integral over $[r,1]$ and we used $p<\alpha$. 
	
	This yields the claimed pointwise bound. (Note that for $x_d\leq 4r$ we have $(1+|x'|^{d+\alpha})^{-1} \sim 1 \wedge |x|^{-d-\alpha}$.) The $L^2$-bound follows by simple integration. The dominant contribution comes from the $x_d$-integral over $[0,1]$.
\end{proof}

\begin{lemma}\label{integral}
	Let $N\geq 1$. Then for all $\beta>0$ and all $a,b\in\R^N$, $r,s>0$,
	$$
	\int_{\R^N} \frac{(rs)^\beta\,dx}{(r^{N+\beta}+|x-a|^{N+\beta})(s^{N+\beta}+|x-b|^{N+\beta})} \lesssim \frac{(r+s)^\beta}{(r+s)^{N+\beta} + |a-b|^{N+\beta}} \,.
	$$
\end{lemma}

\begin{proof}[Proof of Lemma \ref{integral}]
  By symmetry we may assume that $r\leq s$. By translation and dilation, we may and will assume $b=0$ and $s=1$. Thus, it suffices to show
  \begin{align}
    \label{eq:integralaux}
    \int_{\R^N} \frac{r^\beta\, dx}{(r^{N+\beta}+|x-a|^{N+\beta})(1+|x|^{N+\beta})}
    \lesssim \frac{(1+r)^\beta}{(1+r)^{N+\beta} + |a|^{N+\beta}} \,.
  \end{align}
  We distinguish between $|a|\lessgtr(1+r)$.

  \emph{Case $|a|\leq(1+r)$.}
  The right side of \eqref{eq:integralaux} is comparable to $\one_{r<1}+ r^{-N}\one_{r\geq1}$.
  When $r<1$, we bound the left side of \eqref{eq:integralaux} from above by
  \begin{align*}
    r^\beta\int_{\R^N}\frac{dx}{r^{N+\beta}+|x-a|^{N+\beta}}
    = \int_{\R^N}\frac{dx}{1+|x|^{N+\beta}} \sim 1.
  \end{align*}
  When $r\geq1$, we bound the left side of \eqref{eq:integralaux} from above by
  \begin{align*}
    r^\beta\int_{\R^N}\frac{dx}{r^{N+\beta}(1+|x|^{N+\beta})} \sim r^{-N}.
  \end{align*}

  \medskip
  \emph{Case $|a|\geq(1+r)$.}
  The right side of \eqref{eq:integralaux} is comparable to
  $(1+r)^\beta\,|a|^{-N-\beta}$.
  We bound the left side of \eqref{eq:integralaux} from above by
  \begin{align*}
    & \int_{\R^N}\frac{r^\beta}{(r^{N+\beta}+|x-a|^{N+\beta})(1+|x|^{N+\beta})}\left(\one_{|x|<\frac{|a|}{2}} + \one_{\frac{|a|}{2}\leq|x|\leq2|a|} + \one_{|x|\geq2|a|}\right)\,dx \\
    & \quad \lesssim r^\beta\int_{\R^N}\frac{dx}{|a|^{N+\beta}(1+|x|^{N+\beta})} + \int_{\R^N}\frac{r^\beta\one_{|x|\in[|a|/2,2|a|]}}{(r^{N+\beta}+|x-a|^{N+\beta})|x|^{-N-\beta}}\,dx \\
    & \qquad + r^\beta\int_{\R^N}\frac{dx}{|x|^{2N+2\beta}}\one_{|x|\geq2|a|} \\
    & \quad \lesssim r^\beta|a|^{-N-\beta} + |a|^{-N-\beta} r^\beta \int_{\R^N}\frac{dx}{r^{N+\beta}+|x|^{N+\beta}} + r^\beta|a|^{-N-2\beta}
      \lesssim \frac{1+r^\beta}{|a|^{N+\beta}}.
  \end{align*}
  This concludes the proof.
\end{proof}

\begin{proof}[Proof of Lemma~\ref{bdry1}. Case $d=1$ and $\alpha\geq 1$.]
	In fact, we will argue for general $d\geq 1$, assuming \eqref{eq:assv2} and \eqref{eq:asstheta2}. We argue as in explained in the previous section around \eqref{eq:commutator1lalpha}, choosing $\zeta=\theta$.
	
	For the first term in \eqref{eq:commutator1lalpha} we use the bound \eqref{eq:commutator1lalpha1} and note that $[\theta]_{C^1(B_{\ell_x}(x))}$ vanishes unless $x_d\sim r$, in which case it is of order $r^{-1}$. This leads to a bound
	$$
	\one_{x_d\sim r} (1\wedge |x|^{-d-\alpha}) (1\wedge x_d)^{p-\alpha+1} r^{-1} \,. 
	$$
	Similarly, for the second term in \eqref{eq:commutator1lalpha} we use the bound \eqref{eq:commutator1lalpha2} and obtain
	$$
	\one_{x_d\sim r} (1\wedge |x|^{-d-\alpha}) (1\wedge x_d)^{p-\alpha+2} r^{-2} \,.
	$$
	Since $1\wedge x_d \sim r$ for $x_d\sim r$, the two bounds are of the same order.
	
	We now turn to the third term in \eqref{eq:commutator1lalpha}, which we denote by $\widetilde{II}(x)$. We claim that
	\begin{align}\label{eq:bdry2}
		\left| \widetilde{II}(x) \right| & \lesssim \one_{x_d\leq 4r} \int_{y_d>r} \frac{1}{|x'-y'|^{d+\alpha}+y_d^{d+\alpha}} |v(y)| \,dy \notag \\
		& \quad + \one_{x_d>\tfrac r2} \int_{y_d\leq 2r} \frac{1}{|x'-y'|^{d+\alpha}+x_d^{d+\alpha}} |v(y)|\,dy \notag \\
		& \quad + \one_{\tfrac r2<x_d \leq 4r} \frac1r \int_{\tfrac r4<y_d\leq 8r} \frac{|x_d-y_d|}{|x'-y'|^{d+\alpha}+|x_d-y_d|^{d+\alpha}} |v(y)|\,dy \,.
	\end{align}
	This is proved in the exact same way as \eqref{eq:bdry1}.
	
	We now insert the bounds on $v$ into the right side of \eqref{eq:bdry2}. The first two terms are bounded as in the proof of Lemma \ref{bdry1}. The bound for the third term, however, is different now. Noting that $1\wedge|y|^{-d-\alpha} \sim (1+ |y'|^{d+ \alpha})^{-1}$ on the domain of integration, we arrive on the following upper bound on the third term
	$$
	\one_{\tfrac r2<x_d \leq 4r} r^{p-1} \int_{\tfrac r4<y_d\leq 8r} \frac{|x_d-y_d|}{|x'-y'|^{d+\alpha}+|x_d-y_d|^{d+\alpha}} \frac{1}{|y'|^{d+\alpha}+1} \,dy \,.
	$$
	Computing the $y_d$-integral using \eqref{eq:integralxd} and the the $y'$-integral using Lemma \ref{integral}, we can bound
	\begin{align*}
		& \lesssim \one_{\tfrac r2<x_d \leq 4r} r^{p+1} \int_{\R^{d-1}} \frac{1}{|x'-y'|^{d+\alpha}+x_d^{d+\alpha}} \frac{1}{|y'|^{d+\alpha}+1} \,dy' \\ 
		& \lesssim \one_{\tfrac r2<x_d \leq 4r} \frac{r^{p-\alpha}}{1+|x'|^{d+\alpha}} \,.
	\end{align*}
	Combining all these bounds we obtain the claimed pointwise bound. The $L^2$-bound follows as before.
\end{proof}


\subsection{Combined cut-off}\label{sec:cutoffcomb}

We now combine Lemmas \ref{radial1} and \ref{bdry1}.

\begin{corollary}\label{cutoffcomb}
	Let $0<\alpha<2$. Let $0<r\leq 1\leq R<\infty$, assume that $\chi$ and $\theta$ satisfy \eqref{eq:asschi1} and \eqref{eq:asstheta1} and, if $\alpha\geq 1$, also \eqref{eq:asschi2} and \eqref{eq:asstheta2}. Let $\frac{\alpha-1}{2}\leq p<\alpha$, assume that $v$ satisfies \eqref{eq:assv1} and, if $\alpha\geq 1$, also \eqref{eq:assv2} with some $\beta>\alpha-1$. Then
	$$
	\| [(-\Delta)^{\alpha/2},\chi\theta] v \|_{L^2(\R^d_+)} \lesssim r^{p-\alpha+1/2} + R^{-\alpha-d/2} \,.
	$$
\end{corollary}

\begin{proof}
	We decompose
	$$
	\tfrac1{\mathcal A(d,-\alpha)} [(-\Delta)^{\alpha/2},\chi\theta] v(x) = \theta(x) \int_{\R^d}  \frac{\chi(x)-\chi(y)}{|x-y|^{d+\alpha}} v(y)\,dy + \int_{\R^d} \frac{\theta(x)-\theta(y)}{|x-y|^{d+\alpha}} \chi(y)v(y)\,dy
	$$
	and bound the $L^2$-norms of the two terms on the right side separately. For the first term we can drop the term $\theta(x)\in[0,1]$ and apply Lemma \ref{radial1}. For the second term we apply Lemma \ref{bdry1}, noting that the product $\chi v$ satisfies its assumptions. This is clear for $\alpha<1$. For $\alpha\geq1$ we use
	\begin{align*}
          & [\chi v]_{C^\beta(B_r(a))} \\
          & \quad \leq 
            \begin{cases}
              \|\chi\|_{L^\infty(B_r(a))} [v]_{C^\beta(B_r(a))} + \|v\|_{L^\infty(B_r(a))} [\chi]_{C^\beta(B_r(a))} & \text{if}\ \beta\leq 1 \,,\\
              \|\chi\|_{L^\infty(B_r(a))} [v]_{C^\beta(B_r(a))} +\|\nabla v\|_{L^\infty(B_r(a))} [\chi]_{C^{\beta-1}(B_r(a))} & \\
              + \|\nabla\chi\|_{L^\infty(B_r(a))} [v]_{C^{\beta-1}(B_r(a))}
              + \|v\|_{L^\infty(B_r(a))} [\chi]_{C^{\beta}(B_r(a))} & \text{if}\ \beta> 1 \,.
            \end{cases}
	\end{align*}
	All factors involving $\chi$ on the right side are $\lesssim 1$ by \eqref{eq:asschi1}, \eqref{eq:asschi2} and $R\geq 1$. Moreover, we note that if $v$ satisfies \eqref{eq:assv1} and \eqref{eq:assv2} for some $\beta=\beta_0>0$, then it satisfies \eqref{eq:assv2} for any $0<\beta<\beta_0$. We conclude that $\chi v$ satisfies \eqref{eq:assv2} with the same $\beta$ as $v$ does.
\end{proof}


\section{Density of $C_c^\infty(\R^d_+)$}\label{sec:density}

Our goal in this section is to prove the following theorem. It will be the main ingredient to prove the operator core property stated in Theorem \ref{equivalencesobolev}.

\begin{theorem}\label{density}
	Let $\alpha\in(0,2]$ and let $\lambda\geq 0$ when $\alpha<2$ and $\lambda\geq-1/4$ when $\alpha=2$. Let $p$ be defined by \eqref{eq:defp}, and let $s\in(0,2]$. Assume that $s<(1+2p)/\alpha$. Then for any $f\in L^2(\R^d_+)$ there is a sequence $(\phi_n)\subset C^\infty_c(\R^d_+)$ such that
	$$
	L_\lambda^{s/2}\phi_n\to f
	\ \text{in}\ L^2(\R^d_+) \,.
	$$
	If, in addition $f\in\dom L_\lambda^{-s/2}$, then the sequence can be chosen such that, in addition,
	$$
	\phi_n\to L_\lambda^{-s/2}f
	\ \text{in}\ L^2(\R^d_+) \,.
	$$
\end{theorem}

\begin{remark}\label{densityrem}
	Let $\alpha\in(0,2)$, $\lambda\in[\lambda_*,0)$ and assume that $e^{-tL_\lambda}(x,y)$ satisfies the upper bound in \eqref{eq:heatkernel} with $p$ defined by \eqref{eq:defp}. Then Theorem \ref{density} remains valid for this value of $\lambda$. This follows by the same arguments as in the proof below, since Lemma \ref{pointwise} remains valid for this value of $\lambda$.
\end{remark}

Our strategy of proof of this theorem uses some ideas of \cite[Lemma 4.4]{Killipetal2016}. The basic strategy is to first prove Theorem \ref{density} for $f$ of a special form, namely, $f\in L_\lambda^{s/2}e^{-tL_\lambda} C_c^\infty(\R^d_+)$ for some $0<t<\infty$. To do this, we will use the following pointwise bounds on functions in $e^{-tL_\lambda}C_c^\infty(\R^d_+)$. For the definition of the H\"older seminorm see~\eqref{eq:holder}.

\begin{lemma}\label{pointwise}
	Let $\alpha$, $\lambda$ and $p$ be as in Theorem \ref{density}. Let $0<t<\infty$ and $\psi\in e^{-tL_\lambda}C_c^\infty(\R^d_+)$. Then, for all $x\in\R^d_+$,
	\begin{align}
		\label{eq:heatbound1}
		|\psi(x)| & \lesssim (1\wedge x_d)^p(1\wedge |x|^{-d-\alpha}) \,, \\
		\label{eq:heatbound2}
		|L_\lambda\psi(x)| & \lesssim (1\wedge x_d)^p (1\wedge |x|^{-d-\alpha}) \,, \\
		\label{eq:heatbound3}
		|(-\Delta)^{\alpha/2} \psi(x)| & \lesssim (1\wedge x_d)^{p-\alpha} (1\wedge |x|^{-d-\alpha}) \,, \\
		\label{eq:heatbound4}
		[\psi]_{C^\beta(B_{\ell_x}(x))} & \lesssim (1\wedge x_d)^{p-\beta}  (1\wedge |x|^{-d-\alpha})
		\quad \text{with}\ \ell_x:=1\wedge\tfrac{x_d}{2} \,,\ 0<\beta<\alpha \,.
	\end{align} 
\end{lemma}

We remark that for $\alpha=2$ the decay in these bounds can be greatly improved, but it is convenient for us to have a unified statement.

\begin{proof}[Proof of Lemma \ref{pointwise}]
	We write $\psi=e^{-tL_\lambda}k$. The bound \eqref{eq:heatbound1} follows immediately from Theorems \ref{heatkernel} and \ref{localheatkernelhardy}. For the bound \eqref{eq:heatbound2} we write $L_\lambda\psi = e^{-tL_\lambda}L_\lambda k$. For $\alpha=2$ we have $L_\lambda k\in C_c^\infty(\R^d_+)$ and so the claimed bound follows again from Theorem \ref{localheatkernelhardy}. For $0<\alpha<2$ one easily verifies that
	$$
	|L_\lambda k(x)| \lesssim 1\wedge |x|^{-d-\alpha}
	$$
	and then one uses this bound and Theorem \ref{heatkernel} to again deduce \eqref{eq:heatbound2}. We omit the details of this computation.
	
	To prove \eqref{eq:heatbound3} we recall the definition of $\lambda_0$ from Remark \ref{otherdef}. As shown there, we have $(-\Delta)^{\alpha/2} = L_{\lambda_0}$ on functions supported on $\overline{\R^d_+}$. Thus,
	$$
	(-\Delta)^{\alpha/2}\psi = L_\lambda \psi - (\lambda-\lambda_0)x_d^{-\alpha}\psi \,.
	$$
	Therefore \eqref{eq:heatbound3} follows from \eqref{eq:heatbound1} and \eqref{eq:heatbound2}.
	
	Finally, to prove \eqref{eq:heatbound4} we use Schauder estimates. These bounds state that for a function $u$ on $\R^d$, for $a\in\R^d$, $r>0$ and for $0<\beta<\alpha$, one has
	\begin{align}\label{eq:schauder}
		[u]_{C^\beta(B_{r/2}(a))} & \lesssim_{\alpha,\beta,d} r^{-\beta} \| u \|_{L^\infty(B_{2r}(a))} + r^{\alpha-\beta} \| |\cdot - a |^{-d-\alpha} u \|_{L^1(B_{2r}(a)^c)} \notag \\
		& \qquad \ \ + r^{\alpha-\beta} \| (-\Delta)^{\alpha/2} u \|_{L^\infty(B_{2r}(a))} \,.
	\end{align}
	For $\alpha=2$ this bound is classical and can be deduced, for instance, from \cite[Theorem 3.9 and its proof]{GilbargTrudinger2001}. (Indeed, in this case the term involving the norm on $B_{2r}(a)^c$ is not needed.) For $0<\alpha<2$ the bound appears, for instance, in \cite[Corollary 2.5]{RosOtonSerra2014}.
	
	We apply \eqref{eq:schauder} with $a=x$ and $r=2\tilde\ell_x$ with $\tilde\ell_x = 1\wedge \frac{x_d}{8}$. Using \eqref{eq:heatbound1} and \eqref{eq:heatbound3} we easily find that
	\begin{align*}
		\tilde\ell_x^{-\beta} \|\psi\|_{L^\infty(B_{4\tilde\ell_x}(x))} & \lesssim (1\wedge x_d)^{p-\beta} (1\wedge |x|^{-d-\alpha}) \,,\\
		\tilde\ell_x^{\alpha-\beta} \|(-\Delta)^{\alpha/2} \psi\|_{L^\infty(B_{4\tilde\ell_x}(x))} & \lesssim (1\wedge x_d)^{p-\beta} (1\wedge |x|^{-d-\alpha}) \,.
	\end{align*}
	Moreover, using \eqref{eq:heatbound1} a computation whose details we omit shows that
	$$
	\tilde\ell_x^{\alpha-\beta} \| |\cdot - x |^{-d-\alpha} \psi \|_{L^1(B_{4\tilde\ell_x}(x)^c)} \lesssim (1\wedge x_d)^{p-\beta} (1\wedge |x|^{-d-\alpha}) \,.
	$$
	
	Inserting these bounds into \eqref{eq:schauder} we obtain \eqref{eq:heatbound4} with $\tilde\ell_x$ instead of $\ell_x$. The bound with $\ell_x$ follows by a simple covering argument, using for a given $x$ the bound in $B_{\ell_x}(x)$ together with the bounds in $B_{\ell_y}(y)$ for $y\in B_{\ell_x}(x)\setminus B_{\tilde\ell_x}(x)$.
\end{proof}

\begin{proof}[Proof of Theorem \ref{density}]
	\emph{Step 1.} We first prove this theorem for $f$ of a special form, namely, where $f\in L_\lambda^{s/2}e^{-tL_\lambda} C_c^\infty(\R^d_+)$ for some $0<t<\infty$.
	
	Let $0<t<\infty$ and let $\psi\in e^{-tL_\lambda} C_c^\infty(\R^d)$. For parameters $0<r\leq 1\leq R<\infty$ to be determined, we let $\chi$ and $\theta$ be functions as in Corollary \ref{cutoffcomb} and we abbreviate
	$$
	\phi := \chi\theta \psi \,.
	$$
	Then, by \eqref{eq:heatbound1},
	$$
	\|\phi - \psi\|_{L^2(\R^d_+)} \leq \| \one_{x_d\leq 2r} \psi\|_{L^2(\R^d_+)} + \| \one_{|x|> R} \psi\|_{L^2(\R^d_+)} \lesssim r^{p+1/2} + R^{-\alpha -d/2} \,.
	$$
	Moreover,
	$$
	\|L_\lambda(\phi - \psi)\|_{L^2(\R^d_+)} = \|(1-\chi\theta)L_\lambda\psi\|_{L^2(\R^d_+)} + \| [(-\Delta)^{\alpha/2},\chi\theta] \psi \|_{L^2(\R^d_+)} 
	$$
	and, by \eqref{eq:heatbound2},
	$$
	\|(1-\chi\theta)L_\lambda\psi\|_{L^2(\R^d_+)} \lesssim r^{p+1/2} + R^{-\alpha -d/2} \,.
	$$
	For $\alpha<2$ we apply Corollary \ref{cutoffcomb} and find
	$$
	\| [(-\Delta)^{\alpha/2},\chi\theta] \psi \|_{L^2(\R^d_+)}  \lesssim r^{p-\alpha+1/2} + R^{-\alpha-d/2} \,.
	$$
	The same bound holds for $\alpha=2$ as well, as follows by writing
	$$
	[-\Delta,\chi\theta] \psi = -2\nabla(\chi\theta)\cdot\nabla\psi - \Delta(\chi\theta)\psi
	$$
	and using the pointwise bounds \eqref{eq:heatbound1} and \eqref{eq:heatbound4}. Thus, for all $\alpha\leq 2$,
	$$
	\|L_\lambda(\phi - \psi)\|_{L^2(\R^d_+)} \lesssim r^{p-\alpha+1/2} + R^{-\alpha-d/2} \,.
	$$
	Since $0< s\leq 2$ we have, by the spectral theorem,
	$$
	\|L_\lambda^{s/2}(\phi - \psi)\|_{L^2(\R^d_+)} \leq \|\phi - \psi\|_{L^2(\R^d_+)}^{1-s/2} \|L_\lambda(\phi - \psi)\|_{L^2(\R^d_+)}^{s/2} \,.
	$$
	Inserting the above bounds, we conclude that
	$$
	\|L_\lambda^{s/2}(\phi - \psi)\|_{L^2(\R^d_+)} \lesssim r^{p+1/2-\alpha s/2} + R^{-\alpha-d/2} \,.
	$$
	Since, by assumption $s<(1+2p)/\alpha$, this tends to zero as $r\to 0$ and $R\to\infty$. Note also that $\|\phi-\psi\|_{L^2(\R^d_+)}$ tends to zero, proving the second assertion of the theorem for $f=L_\lambda^{s/2}\psi$.
	
	\medskip
	
	\emph{Step 2.} We now prove Theorem \ref{density} in the general case.
	
	Let $f\in L^2(\R^d_+)$ and $\epsilon>0$. By the spectral theorem, we have $e^{-tL_\lambda}f\to f$ as $t\to 0$ and $e^{-tL_\lambda} f\to 0$ as $t\to\infty$. (The latter convergence uses the fact that $0$ is not an eigenvalue of $L_\lambda$.) Therefore, there are $t_1>0$ such that $\|e^{-t_1L_\lambda}f - f\|_{L^2(\R^d_+)}\leq\epsilon$ and $t_2<\infty$ such that $\|e^{-t_2 L_\lambda} f\|_{L^2(\R^d)}\leq\epsilon$. Then, with $t:=t_1/2$ and $T:=t_2/2$,
	$$
	\| (e^{-2tL_\lambda} - e^{-2TL_\lambda})f - f \|_{L^2(\R^d_+)} \leq \|e^{-t_1L_\lambda}f - f\|_{L^2(\R^d_+)} + \|e^{-t_2 L_\lambda} f\|_{L^2(\R^d)}\leq 2\epsilon \,.
	$$
	Since $C_c^\infty(\R^d_+)$ is dense in $L^2(\R^d_+)$ and since $L_\lambda^{-s/2}(e^{-tL_\lambda}-e^{-TL_\lambda})$ is bounded (since $s\in[0,2]$), there is a $k\in C^\infty_c(\R^d_+)$ such that
	$$
	\| k - L_\lambda^{-s/2}(e^{-tL_\lambda}-e^{-TL_\lambda}) f \|_{L^2(\R^d_+)} \leq \epsilon \,.
	$$
	We define $\psi:= (e^{-tL_\lambda}+e^{-TL_\lambda})k$ and write
	$$
	e^{-2tL_\lambda}-e^{-2TL_\lambda} = L_\lambda^{s/2}(e^{-tL_\lambda}+e^{-TL_\lambda}) L_\lambda^{-s/2}(e^{-tL_\lambda}-e^{-TL_\lambda})
	$$
	to find
	\begin{align*}
		\| L_\lambda^{s/2}\psi - f \|_{L^2(\R^d_+)} & \leq \|L_\lambda^{s/2} (e^{-tL_\lambda}+e^{-TL_\lambda})\| \| k -  L_\lambda^{-s/2}(e^{-tL_\lambda}-e^{-TL_\lambda}) f \|_{L^2(\R^d_+)} \\
		& \quad + \| (e^{-2tL_\lambda} - e^{-2TL_\lambda})f - f \|_{L^2(\R^d_+)} \\
		& \leq \left( \|L_\lambda^{s/2} (e^{-tL_\lambda}+e^{-TL_\lambda})\| + 2\right) \epsilon \,.
	\end{align*}
	According to Step 1 (applied both to $L_\lambda^{s/2} e^{-tL_\lambda}k$ and to $L_\lambda^{s/2} e^{-TL_\lambda}k$) there is a $\phi\in C_c^\infty$ such that
	$$
	\| L_\lambda^{s/2} \phi - L_\lambda^{s/2}\psi \|_{L^2(\R^d_+)} \leq \epsilon \,.
	$$
	It follows that
	$$
	\| L_\lambda^{s/2}\phi - f \|_{L^2(\R^d_+)} \leq \left( \|L_\lambda^{s/2} (e^{-tL_\lambda}+e^{-TL_\lambda})\| + 3\right) \epsilon \,.
	$$
	This proves the first assertion of the theorem.
	
	For the second assertion we assume that $f\in\dom L_\lambda^{-s/2}$. Then we choose $t_1,t_2\in(0,\infty)$ such that, in addition, we have $\| e^{-t_1L_\lambda} L_\lambda^{-s/2} f- L_\lambda^{-s/2} f\|_{L^2(\R^d_+)}\!\leq\epsilon$ and $\|e^{-t_2L_\lambda} L_\lambda^{-s/2}f\|_{L^2(\R^d_+)}\!\leq\epsilon$. Then $\|(e^{-2tL_\lambda}-e^{-2TL_\lambda})L_\lambda^{-s/2} f - L_\lambda^{-s/2} f\|_{L^2(\R^d_+)} \leq 2\epsilon$. Moreover, by Step 1, we may assume, in addition, that $\|\phi - \psi\|_{L^2(\R^d_+)}\leq\epsilon$. From this one deduces, similarly as before,
	$$
	\| \phi - L_\lambda^{-s/2} f \|_{L^2(\R^d_+)} \leq \left( \|e^{-tL_\lambda}+e^{-TL_\lambda}\| + 3\right) \epsilon \,,
	$$
	which completes the proof of the theorem.	
\end{proof}


\section{Proof of the main result}\label{sec:proofmain}

\begin{proof}[Proof of Theorem \ref{equivalencesobolev}]
	We begin with the proof of \eqref{eq:equivalencesobolev1} and \eqref{eq:equivalencesobolev2} for functions $u\in C^\infty_c(\R^d_+)$. Using just the triangle inequality, the claims are an immediate consequence of the usual Hardy inequality, as well as its reversed and generalized versions in Theorems~\ref{genhardy} and \ref{reversehardy}. The argument is as in \cite{Killipetal2018,Franketal2021} and we omit the details.
	
	We now extend \eqref{eq:equivalencesobolev1} to all $u\in\dom L_\lambda^{s/2}$. According to Theorem \ref{density} (applied to $f=L_\lambda^{s/2}u$) there is a sequence $(\phi_n)\subset C_c^\infty(\R^d_+)$ such that $\phi_n\to u$ in $L^2(\R^d_+)$ and $L_\lambda^{s/2} \phi_n\to L_\lambda^{s/2} u$ in $L^2(\R^d_+)$. It follows from inequality \eqref{eq:equivalencesobolev1}, applied to $\phi_n-\phi_m$, that $(L_0^{s/2}\phi_n)$ is Cauchy in $L^2(\R^d_+)$ and therefore convergent to some $f\in L^2(\R^d_+)$. Since the operator $L_0^{s/2}$ is closed, we conclude that $u\in\dom L_0^{s/2}$ and $L_0^{s/2}u =f$. The claimed inequality \eqref{eq:equivalencesobolev1} for $u$ now follows by passing to the limit in the inequality for $\phi_n$.
	
	The extension of \eqref{eq:equivalencesobolev2} follows similarly. We only note that the $p$ that corresponds to $\lambda=0$ is $(\alpha-1)_+$. Therefore the assumption $s<(1+2(\alpha-1)_+)/\alpha$ in Theorem~\ref{equivalencesobolev} coincides with the assumption in Theorem~\ref{density} (applied with $\lambda=0$).
\end{proof}

We now discuss optimality of the assumptions in Theorem \ref{equivalencesobolev}.

\begin{proposition}\label{opt}
	Let $\alpha\in(0,2]$ and let $\lambda\geq0$ when $\alpha\in(0,2)$ and $\lambda\geq-1/4$ when $\alpha=2$. Let $p$ be defined by \eqref{eq:defp}, and let $s\in(0,2]$.
	\begin{enumerate}
		\item[(1)] If $\lambda<0$, $p<d-1/2$ and $\dom (L_\lambda^{(\alpha)})^{s/2} \subset \dom (L_0^{(\alpha)})^{s/2}$, then $s<(1+2p)/\alpha$.
		\item[(2)] If $\lambda>0$, $(\alpha-1)_+<d-1/2$ and $\dom (L_0^{(\alpha)})^{s/2} \subset (L_0^{(\alpha)})^{s/2}$, then $s<(1+2(\alpha-1)_+)/\alpha$.
	\end{enumerate}
\end{proposition}

Note that the `additional' assumptions $p<d-1/2$ and $(\alpha-1)_+<d-1/2$ are automatically satisfied when $d\geq 2$ or when $d=1$ and $\alpha\leq 3/2$.

\begin{remark}
	Let $\alpha\in(0,2)$, $\lambda\in[\lambda_*,0)$ and assume that $e^{-tL_\lambda}(x,y)$ satisfies the lower bound in \eqref{eq:heatkernel} with $p$ defined in \eqref{eq:defp}. Then part (1) in Proposition \ref{opt} remains valid for this value of $\lambda$. This follows by the same arguments as in the proof below.
\end{remark}

\begin{proof}
	We prove part (1), the other part being proved similarly. We will prove the theorem under the additional assumption $s<\frac{1+2(\alpha-1)_+}{2}\wedge\frac{2d}{\alpha}$. Note that the assumption $\lambda<0$ (which is equivalent to $p<(\alpha-1)_+)$ and the assumption $p<d-1/2$ imply that the interval $[\frac{1+2p}\alpha,\frac{1+2(\alpha-1)_+}{2}\wedge\frac{2d}{\alpha})$ is nonempty. Thus our proof will show that in this interval the inclusion $\dom L_\lambda^{s/2}\subset\dom L_0^{s/2}$ fails. By operator monotonicity of taking roots (see, e.g., \cite[Section~10.4]{BirmanSolomjak1987}) it then follows that the inequality also fails for all $s\geq \frac{1+2(\alpha-1)_+}{2}\wedge\frac{2d}{\alpha}$.
	
	Thus, assume that $s<\frac{1+2(\alpha-1)_+}{2}\wedge\frac{2d}{\alpha}$. Let $u\in e^{-L_\lambda} C_c^\infty(\R^d_+)$. Then $u\in\dom L_\lambda^{s/2}\subset \dom L_0^{s/2}$. Applying Theorem \ref{genhardybdd} with $\lambda=0$ and $g=L_0^{s/2} u$ (here we need the upper bound on $s$) we infer that $x_d^{-\alpha s/2} u\in L^2(\R^d_+)$. Using the lower bound in Theorem \ref{heatkernel} and arguing as in the necessity part of the proof of Theorem \ref{genhardybdd} we deduce that $s<(1+2p)/\alpha$, as claimed.
\end{proof}

\appendix

\section{Definition of the exponent $p$}\label{s:defp}

Throughout this appendix we assume $\alpha\in(0,2)$. For $p\in(-1,\alpha)$, we set
\begin{align*}
	\gamma(\alpha,p) := \int_0^1 \frac{(t^p-1)(1-t^{\alpha-p-1})}{(1-t)^{1+\alpha}}\,dt \,.
\end{align*}
The function $C$ is defined in \cite[Remark 3.3]{Choetal2020} for $d=1$ by
\begin{align*}
	\begin{split}
		(-1,\alpha)\ni p \mapsto C(p)
		& := \mathcal A(1,-\alpha)\, \gamma(\alpha,p)
	\end{split}
\end{align*}
and in \cite[Equation (3.4)]{Choetal2020} for $d\geq 2$ by
\begin{align*}
	\begin{split}
		(-1,\alpha)\ni p \mapsto C(p)
		& := \mathcal A(d,-\alpha)\, \frac{|\Sph^{d-2}|}2\, B\left(\frac{\alpha+1}2,\frac{d-1}2\right) \gamma(\alpha,p)
	\end{split}
\end{align*}
with the beta function $B$. Let us show that that these definitions coincides with our definition \eqref{eq:defcp} and, in particular, that they are independent of $d$. 

First, we recalling the formula for $\mathcal A(d,-\alpha)$ from \eqref{eq:adalpha} and $|\Sph^{d-2}|= \frac{2\pi^{\frac{d-1}2}}{\Gamma(\frac{d-1}{2})}$ we find
\begin{equation*}
	\mathcal A(d,-\alpha)\, \frac{|\Sph^{d-2}|}2\, B\left(\frac{\alpha+1}2,\frac{d-1}2\right)
	= \mathcal A(1,-\alpha) \,,
\end{equation*}
which already shows the independence of $d$. Thus, from now on $d=1$. Moreover, by the reflection and duplication formulas of the gamma function, we obtain
\begin{equation}
	\label{eq:compprefac}
	\mathcal A(1,-\alpha) = \frac{\sin\tfrac{\pi\alpha}{2}}{\pi}\ \Gamma(\alpha+1) \,.
\end{equation}

Next, according to \cite[(2.2)]{BogdanDyda2011} we have for $\alpha\neq 1$ and $\alpha>p>-1$
\begin{align*}
	\begin{split}
		\gamma(\alpha,p)
		& = \frac{1}{\alpha(\alpha-1)}\left[(p+1-\alpha)(p+2-\alpha) B(p+1,2-\alpha) \right. \\
		& \qquad\qquad\quad \left. - (1-\alpha)(2-\alpha)B(1,2-\alpha) + p(p-1) B(\alpha-p,2-\alpha)\right].
	\end{split}
\end{align*}
Expressing the beta functions as gamma functions and using its functional equation, we find
\begin{align}
	\label{eq:computegamma}
	\begin{split}
          \gamma(\alpha,p)
		& = \frac{1}{\alpha(\alpha-1)}\left[\frac{\Gamma(p+1)\,\Gamma(2-\alpha)}{\Gamma(p-\alpha+1)} - (1-\alpha) + \frac{\Gamma(\alpha-p)\,\Gamma(2-\alpha)}{\Gamma(-p)} \right] \\
		& = \frac1\alpha - \frac{\Gamma(1-\alpha)} \alpha \left[ \frac{\Gamma(p+1)}{\Gamma(p-\alpha+1)} + \frac{\Gamma(\alpha-p)}{\Gamma(-p)} \right].
	\end{split}
\end{align}
Using the reflection formula for the gamma function, we obtain
$$
\frac{\Gamma(p+1)}{\Gamma(p-\alpha+1)} + \frac{\Gamma(\alpha-p)}{\Gamma(-p)}
	= - \frac1\pi \Gamma(1+p)\Gamma(\alpha-p) \left( \sin\pi(p-\alpha) +\sin\pi p \right).
$$
Inserting this into \eqref{eq:computegamma} and combining it with \eqref{eq:compprefac} we obtain
$$
C(p) = \frac{\sin\tfrac{\pi\alpha}{2}}{\pi}\, \Gamma(\alpha) +  \Gamma(\alpha)\Gamma(1-\alpha)\Gamma(1+p)\Gamma(\alpha-p) \frac{\sin\tfrac{\pi\alpha}{2}}{\pi^2} \left( \sin\pi(p-\alpha) +\sin\pi p \right).
$$
The claimed formula \eqref{eq:defcp} now follows from  $\Gamma(\alpha)\Gamma(1-\alpha) = \pi/\sin(\pi\alpha)$ (by the reflection formula) and
$$
\frac{\sin\tfrac{\pi\alpha}{2}}{\sin \pi\alpha} \left( \sin\pi(p-\alpha) +\sin\pi p \right) = \sin\frac{\pi(2p-\alpha)}{2} \,.
$$

\medskip

Having established the equality between our definition of $C$ and that in \cite{Choetal2020}, we can use its properties established in \cite[Subsection 3.1]{Choetal2020}, namely, its strict monotonicity on $[\frac{\alpha-1}{2},\alpha)$, its divergence at $p=\alpha$ and its vanishing at $p=\alpha-1,0$. Its symmetry with respect to $p=\frac{\alpha-1}2$ is immediate from \eqref{eq:defcp}.


\section{Proof of Theorem \ref{localheatkernelhardy}}\label{s:localheatkernel}

Throughout this section, we assume $\alpha=2$. 

\emph{Step 1.} By separation of variables, we have
$$
e^{-t L_\lambda}(x,y) = e^{t\Delta_{\R^{d-1}}}(x',y') \cdot e^{-t(-\Delta_{\R_+} + \lambda x_d^{-2})}(x_d,y_d) \,.
$$
Since the first factor is equal to $(4\pi t)^{-\frac{d-1}2} e^{-|x'-y'|^2/4t}$, we see that the theorem in dimensions $d\geq 2$ follows from its special case for $d=1$. Thus, in what follows we consider the latter case.

\medskip

\emph{Step 2.} It is convenient to work with a unitarily equivalent version of $L_\lambda$. Namely for $\mu\geq 0$ we consider the nonnegative quadratic form
$$
\int_0^\infty |u'|^2 r^{2\mu+1}\,dr
$$
defined for $u\in C^1_c(\R_+)$. By a theorem of Friedrichs this form gives rise to a selfadjoint, nonnegative operator $\mathcal L_\mu$ in the Hilbert space $L^2(\R_+,r^{2\mu+1}dr)$. We note that functions $u$ in the operator domain of $\mathcal L_\mu$ are twice weakly differentiable and $\mathcal L_\mu u = - u'' - (2\mu+1)r^{-1} u'$.

It is well-known that the operator $\mathcal L_\mu$ can be diagonalized by a Hankel transform, which, in particular, gives an integral formula for its heat kernel. The resulting integral over Bessel functions can be carried out using standard formulas and one arrives at the explicit expression
\begin{equation}
  \label{eq:heathalflineexplicit}
  \me{-t\cL_{\mu}}(r,s) = (2t)^{-1}\left(\frac{1}{rs}\right)^\mu \cdot \exp\left(-\frac{r^2+s^2}{4t}\right) I_{\mu}\left(\frac{rs}{2t}\right) \\
\end{equation}
This formula appears, for instance, in \cite[p.~75]{BorodinSalminen2002}. We emphasize that this is the heat kernel with respect to the underlying measure $r^{2\mu+1}\,dr$, that is
$$
(e^{-t\mathcal L_\mu}f)(r) = \int_0^\infty \me{-t\cL_{\mu}}(r,s) f(s)\,s^{2\mu+1}\,ds \,.
$$
Using the facts that
$$
\lim_{z\to 0} z^{-\mu} I_\mu(z) = \frac{2^{-\mu}}{\Gamma(1+\mu)}
\qquad\text{and}\qquad
\lim_{z\to\infty} z^{\frac12} e^{-z} I_\mu(z) = \frac{1}{\sqrt{2\pi}} \,,
$$
we immediately obtain from \eqref{eq:heathalflineexplicit} that
\begin{align}
	\label{eq:bogusmalecki}
	\begin{split}
		\me{-t\cL_\mu}(r,s)
		\sim \left(1\wedge \frac{r\cdot s}{t}\right)^{\mu+\frac12}\left(\frac{1}{rs}\right)^{\mu+\frac12}\cdot t^{-\frac12}\cdot\exp\left(-\frac{(r-s)^2}{4t}\right).
	\end{split}
\end{align}

Next, we show that there is a $0<c< 1$ such that for all $r,s,t>0$ one has
\begin{align}
	\label{eq:bogusmalecki2}
	\begin{split}
		& \left(1\wedge \frac{r}{\sqrt t}\right)^{\mu+\frac12}\left(1\wedge \frac{s}{\sqrt t}\right)^{\mu+\frac12} \left(\frac{1}{rs}\right)^{\mu+\frac12} \cdot t^{-\frac12}\cdot\exp\left(-\frac{(r-s)^2}{4t}\right)\\
		& \quad \lesssim \me{-t\cL_\mu}(r,s)\\
		& \quad \lesssim \left(1\wedge \frac{r}{\sqrt t}\right)^{\mu+\frac12}\left(1\wedge \frac{s}{\sqrt t}\right)^{\mu+\frac12} \left(\frac{1}{rs}\right)^{\mu+\frac12} \cdot t^{-\frac12}\cdot\exp\left(-c \frac{(r-s)^2}{4t}\right)\,.
	\end{split}
\end{align}
In fact, we show that this holds for any $0<c<1$, but the constant that our proof gives for the second ``$\lesssim$'' diverges as $c$ approaches $1$.

Note that $(1\wedge \frac{r}{\sqrt t})(1\wedge \frac{s}{\sqrt t})$ and $1\wedge \frac{rs}{t}$ coincide when either $r,s\leq\sqrt t$ or $r,s\geq \sqrt t$ and that the former is never larger than the latter for any $r,s$. In view of \eqref{eq:bogusmalecki}, this proves the first ``$\lesssim$'' in \eqref{eq:bogusmalecki2} and shows that we only need to prove the second ``$\lesssim$'' in the regions $r\leq\sqrt t\leq s$ and $s\leq\sqrt t\leq r$. By symmetry, it suffices to consider the former region. Moreover, by scaling, we can suppose $t=1/4$. We abbreviate $p:=\mu+1/2$ and show that there is a $0<c<1$ such that, for all $r\leq 1/2\leq s$,
\begin{align}
	\label{eq:bogusmaleckiaux1}
	\begin{split}
		(1\wedge rs)^p \exp\left(-(r-s)^2\right)
		\lesssim r^p \cdot \exp\left(-c(r-s)^2\right).
	\end{split}
\end{align}
This will clearly imply the second ``$\lesssim$'' in \eqref{eq:bogusmalecki2}.

For the proof of \eqref{eq:bogusmaleckiaux1} we distinguish between $rs\leq1$ and $rs\geq1$ and start with the former case. Here we need to show
\begin{align*}
	s^p\exp(-(r-s)^2) \lesssim \exp(-c(r-s)^2)\,.
\end{align*}
This can be inferred by taking the $p$-th root and the inequalities $s=(s-r)+r\leq(s-r)+1/2$. The term corresponding to $(s-r)$ can be controlled by taking $c<1$ arbitrary. To prove \eqref{eq:bogusmaleckiaux1} when $rs\geq1$ we need to show
\begin{align*}
	\exp\left(-(r-s)^2\right)
	\lesssim r^p \cdot \exp\left(-c(r-s)^2\right)\,.
\end{align*}
This can be inferred by multiplying by $r^p$, taking the $p$-th root and the inequalities $r^{-1}\leq s = (s-r)+r\leq (s-r)+1/2$. As before, the term corresponding to $(s-r)$ can be controlled by taking $c<1$ arbitrary. This completes the proof of \eqref{eq:bogusmaleckiaux1}.

\medskip

\emph{Step 3.} It remains to translate the result from the operator $\mathcal L_\mu$ to the operator $L_\lambda$. The operator $U$, defined by $(Uf)(x) = x^{\mu+\frac12} f(x)$, is unitary from $L^2(\R_+,r^{2\mu+1}\,dr)$ to $L^2(\R_+,dx)$. It maps $C^1_c(\R_+)$ into itself and, for a function $u$ from this space, we find by an integration by parts
$$
\int_0^\infty \left( |(Uu)'(x)|^2 + (\mu^2-\tfrac14) x^{-2} |Uu(x)|^2\right)dx = \int_0^\infty |u'(r)|^2 r^{2\mu+1}\,dr \,.
$$
This implies that
$$
U^* L_{\mu^2-1/4} U = \mathcal L_\mu
$$
and, consequently, for all $t,x,y>0$,
$$
e^{-t L_{\mu^2-1/4}}(x,y) = (xy)^{\mu+\frac12} e^{-t\mathcal L_\mu}(x,y) \,.
$$
In view of \eqref{eq:bogusmalecki2} we obtain the assertion in Theorem \ref{localheatkernelhardy}.\qed


\newcommand{\etalchar}[1]{$^{#1}$}
\def\cprime{$'$}

\end{document}